\documentclass[a4paper,10pt]{article}

\usepackage[T1]{fontenc}
\usepackage[utf8]{inputenc}
\usepackage{amssymb,amsmath,amsthm,amsfonts,mathrsfs,latexsym, hyperref, color}
\usepackage{graphicx,enumerate}
\usepackage[bf,big,raggedright,indentafter,pagestyles]{titlesec}

\renewcommand\theequation{\thesection.\arabic{equation}}
\newtheorem*{step}{Step}
\newtheorem*{proof1}{Proof of Proposition \ref{PROPGM}}
\newtheorem*{proof2}{Proof of Proposition \ref{contv}}
\newtheorem*{proof3}{Proof of Lemma \ref{LEMcont}}

\newtheorem*{proofH3}{Proof of Proposition \ref{PROPH3}}
\newtheorem{thrm}{Theorem}[section]
 \newtheorem{dfntn}[thrm]{Definition}
 \newtheorem{rmrk}[thrm]{Remark}
 \newtheorem{prpstn}[thrm]{Proposition}
 \newtheorem{lmm}[thrm]{Lemma}

\begin{document}
\title{Junction conditions for finite horizon optimal control problems on multi-domains with continuous and discontinuous solutions}
\author{
Daria Ghilli\thanks{Institute of Mathematics and Scientific Computing, Karl-Franzens University of Graz, Heinrichstra{\ss}e 36, 80101, Graz, Austria. Email: daria.ghilli@uni-graz.at} \and
Zhiping Rao\thanks{RICAM, Austrian Academy of Sciences, Altenbergerstra{\ss}e 69, A-4040 Linz, Austria.
Email: zhiping.rao@ricam.oeaw.ac.at} \and
Hasnaa Zidani\thanks{Unit\'e de Math\'ematiques Appliqu\'ees, ENSTA ParisTech, 828  Boulevard des Mar\'echaux, 91762 Palaiseau Cedex. Email: Hasnaa.Zidani@ensta-paristech.fr}
}
%

%
%
%
%


%

\newcommand{\R}{\mathbb{R}}
\newcommand{\Z}{\mathbb{Z}}
\newcommand{\N}{\mathbb{N}}
\newcommand{\I}{\mathbb{I}}
\newcommand{\M}{\mathcal{M}}
\newcommand{\D}{\mathcal{D}}
\newcommand{\Ms}{\mathscr{M}}
\newcommand{\C}{\mathcal{C}}
\newcommand{\Hy}{\mathcal{H}}
\newcommand{\Proj}{\mathcal{P}}
\def\E{\mathcal {E}}
\newcommand{\e}{\varepsilon}
\newcommand{\s}{\sigma}
\newcommand{\vp}{\varphi}
\newcommand{\tv}{\vartheta}
\newcommand{\btv}{\bar{v}}
\newcommand{\K}{\mathcal{K}}
\newcommand{\bK}{\mathbb{K}}
\newcommand{\W}{\mathcal{W}}
\newcommand{\T}{\mathcal{T}}
\newcommand{\ra}{\rightarrow}
\newcommand{\aA}{a\in\mathcal{A}}
\newcommand{\Om}{\Omega}
\newcommand{\oO}{\mathop{\Om}\limits^{\circ}}
\newcommand{\A}{\mathcal{A}}
\newcommand{\As}{\mathscr{A}}
\newcommand{\xb}{\bar{x}}
\newcommand{\yb}{\bar{y}}
\newcommand{\Fe}{F_{\e}}
\newcommand{\Fn}{F^{new}}
\newcommand{\Ft}{\widetilde{F}}
\newcommand{\Gt}{\widetilde{G}}
\newcommand{\Ep}{\mathcal{E}p}
\newcommand{\Hp}{\mathcal{H}p}

\def\gr{\text{gr }}
\def\dist{\text{dist}}
\def\co{\text{co }}
\def\ae{\text{a.e. }}
\def\Esc{\text{Esc }}
\def\proj{\text{proj}}
\def\bdry{\text{bdry }}
\def\epi{\text{epi }}
\def\dom{\text{dom }}
\def\ext{\text{ext}}
\def\cone{\text{cone }}
\def\iint{\text{int }}
\def\ri{\text{\em{r-int} }}
\def\rb{\text{\em{r-bdry} }}
\def\meas{\text{\em{meas}}}
\def\Klim{\text{\em{Klim}}}
\maketitle

\begin{abstract}
This paper deals with junction conditions for Hamilton-Jacobi-Bellman (HJB) equations for finite horizon control problems on multi-domains.  We consider two different cases where the final cost is continuous or  lower semi-continuous. In the continuous case  we  extend the results in \cite{RZ} in a more general framework with switching running costs and weaker controllability assumptions. The comparison principle has been established to guarantee the uniqueness and the stability results for the HJB system on such multi-domains. In the lower semi-continuous case, we characterize the value function as the unique lower semi-continuous viscosity solution of the HJB system, under a local controllability assumption. 
\end{abstract}
\section*{Introduction}
We study finite horizon optimal control problems on multi-domains of $\R^d$ with interfaces where the dynamics and the cost functions may have discontinuities. In particular, we consider a cellular partition of  $\R^d$, that is, a disjoint union of subdomains $\Omega_i,\,\, i=1,\cdots,  m,$ where the interfaces coincide with  crossing hyperplanes separating the subdomains.  The goal of our investigation is to identify the junction conditions on the interfaces such that the optimal control problem involving the trajectories switching between the subdomains or staying on the interfaces is well defined and the associated Hamilton-Jacobi-Bellman (HJB) equation has a unique solution.

The discontinuous setting across the interfaces leads us to the study  state-discontinuous Hamilton-Jacobi equations. The viscosity notion of solutions to HJ equations was firstly extended in the discontinuous case in \cite{I}, providing the first vision on this subject. Then, several attention has been given to the type of conditions one has to add in order to establish the comparison principle. In \cite{S}, a class of stationary HJ equations with discontinuous Lagrangian has been studied and an uniqueness result is provided using a special structure of the discontinuities. Later, the viscosity notion was extended in \cite{CS} to the case where the Hamiltonian is state-measurable, and a comparison principle is obtained under an adequate assumption which avoid complex interactions between the trajectories and the interfaces. 

Control problems on multi-domains  has become an active field of investigation and several papers have been particularly influential for our work. The first paper on stratified domains investigating the HJB tangential equations on the interfaces has been the work \cite{BH} by Bressan and Hong,  where a rather complete analysis of discontinuous deterministic control problems in stratified domains has been carried out.  Then, in \cite{BBC,BBC2}, both the infinite horizon and finite horizon problems on two-domains are studied. In both works, the authors consider different types of strategies for the trajectories to identify the proper HJB equations to provide maximal and minimal solutions and the conditions for uniqueness. The controllability is assumed in the whole space in \cite{BBC}, and then has been weakened to a normal controllability with respect to the interface in \cite{BBC2}. Stability results are also provided.  Then, following a similar approach and under similar controllability assumptions to \cite{BBC, BBC2}, a rather general class of discontinuous deterministic control problems on stratified domains have been studied out in \cite{BC}. \\
The work \cite{BW} has particularly attracted our attention by providing a selection principle for the dynamics on the interfaces of stratified domains, called {\em essential dynamics}, to obtain invariance properties. By following this selection principle,  junctions conditions on the interfaces on multi-domains are provided in \cite{RZ}, where the characterization result is carried out under a full controllability assumption. The further work \cite{RSZ} consider an infinite horizon problem in two-domains under a weaker controllability assumption and a convexity assumption for the set of dynamics/costs.  Finally, we would like to mention some recent work on networks \cite{ACCT,IMZ,IM,CM} which share the same kind of difficulty as this subject.\\
A  different framework is considered in \cite{HZ}, where  an infinite horizon state constrained control problem with a constraint set having a stratified structure is studied. We refer also to \cite{HZW}, where the same approach has been used to study the minimum time problem and the Mayer problem for stratified state constraints. In such situations,  the interior of the set may be empty and the
classical pointing qualification hypothesis to guarantee the characterization of the continuous value function are not relevant. Then, the discontinuous value function
is characterized by means of a system of HJB equations on each stratum that composes
the state constraints. This result is obtained under a local controllability assumption which is
required only on the strata where some chattering phenomena could occur.\\

In the present work,  we consider first the case of a discontinuous control problem on a stratification of $\R^d$ with continuous final cost. In this  case, the value function is characterized as the unique continuous solution of the set of HJB equations  modeling the control problem  coupled with the junction conditions on the interfaces.  Following the concept of essential dynamics introduced in \cite{BW}, HJB junction equations on the interfaces are provided, the viscosity notion for the HJB system is introduced and, under some controllability conditions, the comparison result and the existence and uniqueness of a continuous solution are obtained.  In this continuous setting, we develop further the ideas introduced in \cite{RZ,RSZ}, but with some significant contribution. In comparison to \cite{RZ, RSZ}, the present work considers a more general structure of multi-domains with  crossing hyperplanes involving switching running costs and under weaker controllability and convexity assumptions.   Moreover in our framework both the dynamic and the cost can be unbounded, differently from \cite{RZ, RSZ, BBC, BBC2, BC}. Another technical issue is the convexity condition for the set of velocities and costs. As in \cite{BH}, we assume a weaker convexity hypothesis than the one in \cite{BBC,BBC2}. The advantage of this assumption is to include more general cases and to avoid working with the relaxed problems. Finally, we remark that in \cite{BBC, BBC2, BC} the comparison result and consequently the continuity of the value function is proven under a normal controllability condition and the results are obtained under mainly PDE techniques combined with some control arguments. In the present work we suppose also a tangential controllability condition only on the interfaces, leading to a Lipschitz regularity of the value function on the interfaces. Also, the techniques used are based mainly on control theory.
We provide  a stability result, which on the other hand is based directly on the viscosity notion without  control arguments. \\
In the second part of our work, we consider the case of a lower semi-continuous terminal cost which, as far as we know, has not been considered previously in the framework of multi-domains control problems. In this setting, the value function is characterized as the unique lower semi-continuous bilateral solution of the set of HJB equations coupled with  tangential junction conditions on the interfaces. Our approach is inspired to \cite{HZ, HZW}, by adapting  to our setting the techniques used there in the different framework of a stratified state constraint control problem. As in the continuous case, we consider an finite horizon control problem, whereas in \cite{HZ, HZW} respectively the infinite horizon case and the Mayer problem are studied. We assume a controllability assumption which includes the case in which there is no controllability anywhere on the interfaces and in particular allows us to treat the arising of some chattering phenomena.\\
The paper is organized as follows. In Section \ref{nsmres} we set some notations and we state our main results. In Section \ref{propvf} we recall some main properties of the value function. Section \ref{supsup} and Section \ref{subsub} are devoted to the characterization of super and subsolutions through super and sub-optimality principles. In Section \ref{mres}  we prove are main results, namely existence, uniqueness. Finally, Section \ref{sres} is devoted to the stability result.

\section{Notations, setting of the problem and main results}\label{nsmres}
\subsection{Notations}
For any $\M$ subset of $\R^d$, the closure of $\M$ is denoted as $\overline \M$. For each $x\in\R^d$, $\|x\|$ denotes the Euclidean norm of $x$ and $d(x,\M)$ denotes the distance from $x$ to $\M$, i.e.
 \[
 d(x,\M)=\inf\{\|x-z\|\,:\,z\in\M\}.
 \]
In the sequel, for any function $w:\R^p\ra\R$, $\Ep(w)$ and $\Hp(w)$ denote respectively the epigraph and hypograph of $w$, i.e.
\[
\Ep(w):=\{(x,z)\in\R^p\times\R\,:\,w(x)\leq z\},\ \Hp(w):=\{(x,z)\in\R^p\times\R\,:\,w(x)\geq z\}.
\]
If $V_k$ is an affine subspace of $\mathbb{R}^d$ of dimension $k$, we denote by $V_k^{\bot} $ the subspace of $\mathbb{R}^d$ such that the following decomposition holds $\mathbb{R}^d=V_k \bigoplus V_k^{\bot}$.
\subsection{Assumptions}\label{secass}

We consider a cellular decomposition of $\R^d$ into $m$-cells $\{\Omega_i\}_{i=1,\cdots, m}$, separated by hyperplanes $(\Hy_j)_{j=1,\cdots,q}$, such that for all $j_1,j_2 \in \{1,\cdots, q\}$, $ \Hy_{j_1} \neq \Hy_{j_2}$ for $j_1 \neq j_2$ and  either $\Hy_{j_1} //\Hy_{j_2} $ or $\Hy_{j_1} \perp \Hy_{j_2}$. Set $\Gamma:=\bigcup_{j=1,\cdots,q} \Hy_j$.  More specifically, we assume under the above notations:
\[
{\bf (H1)}
\begin{cases}
\text{(i)}\;  &\R^d=\Gamma\bigcup\left(\cup_{i=1}^m\Om_i\right), \\
\text{(ii)}\; &\Gamma\cap\Om_i=\emptyset \, \, \forall i=1,\cdots, m, \\
\text{(iii)}\; &\Om_i \mbox{ is open and connected}.\\
\end{cases}
\]
Throughout the paper, we use the following notation: 
\[
\Gamma=\bigcup_{k=1}^{p }\Gamma_k, \quad (\Gamma_k)_{k=1,\cdots,p} \mbox{ are pairwise disjoint}.
\] 
 Moreover, we denote by
$(\Gamma^0_k)_{k}$ the subdomains of $\Gamma$ such that there exist two  hyperplanes $\Hy_{k_1}, \Hy_{k_2}, k_1 \neq k_2$ such that $\Gamma^0_k= \Hy_{k_1}\cap \Hy_{k_2}$.\\
Moreover, we will occasionally denote  by $\M_k$ either $\Omega_k$ either $\Gamma_k$, so that
$$
\R^d=\bigcup_{k=1}^{l+m} \M_k, \quad \mbox{where  for each } k: \quad \M_k=\Omega_k \quad \mbox{or } \M_k=\Gamma_k.
$$
Note that  the set of interfaces separating the cells in our partition consists in  the two set of parallel  hyperplanes.  For the rest of the paper, we make the arbitrary choice of choosing a unique direction for the exterior normal vector for each of these two sets.  We denote the normal to each $\Hy_j$ with this chosen direction by  $\vec{n}_j$.

We are given a control problem on $\R^d$ with dynamic  $f\, : \, \mathbb{R}^d \times \As \mapsto \mathbb{R}^d$  and running cost $\ell\, : \, \mathbb{R}^d \times \As  \mapsto \mathbb{R}$, where $\As $ is a compact set of $\R^n$. For simplicity, throughout the paper we will consider  also the following  multifunctions  notations $F:\R^d\rightsquigarrow\R^d, L:\R^d\rightsquigarrow\R^d$
\begin{itemize}
\item 
$
F(x):=\{f(x,a)\,:\,a \in \As )\}), \quad L(x):=\{\ell(x,a)\,:\,a\in \As \});
$
\item for any $i=1, \ldots, m$, $F_i=F|_{\overline{\Omega}_i}$, $L_i=L|_{\overline{\Omega}_i}$.
\end{itemize}
We assume the following standard hypothesis on $F$ and $L$:
\[
{\bf (HF)}
\begin{cases}
\text{(i)}\; & \hspace{-0.4cm}  x\rightsquigarrow F(x) \mbox{ has non-empty compact images and  is upper semi-continuous}\footnotemark;\\
\text{(ii)}\; & \hspace{-0.4cm} \forall i \in \{1,\ldots m\} \mbox{ the map } \,\, x\rightsquigarrow F_i(x) \mbox{ is locally Lipschitz continuous }\mbox{ w.r.t the Hausdorff distance};\\
\text{(iii)}\; & \hspace{-0.4cm} \mbox{ There exists } c_f>0 \mbox{ such that }\ \ \max\{|p|\, |\, p \in F(x)\}\leq c_f(1+|x|);
\end{cases}
\]
\footnotetext{We recall that  a multifunction $x \rightsquigarrow F(x)$ is said to be upper-semi
continuous at $x_0$ if for any open set $\mathcal{C} \supset F(x_0)$, there exists an open set $\omega$ containing $x_0$ such that
$F(\omega) \subset \mathcal{C}$. In other terms, $F(x) \supset \limsup_{y\to x} F(y).$}
\[
{\bf (HL)}
\begin{cases}
\text{(i)} & \hspace{-0.4cm} x \rightsquigarrow L(x) \mbox{ has non-empty compact images and is upper semi-continuous}\footnotemark[1];\\
\text{(ii)} & \hspace{-0.4cm} \forall i \in \{1,\ldots m\} \mbox{ the map } \, x\rightsquigarrow L_i(x) \mbox{ is locally Lipschitz continuous } \mbox{ w.r.t. the Hausdorff distance}.\\
\text{(iii)} & \hspace{-0.4cm} 

\mbox{ There exists } c_l>0 \mbox{ and } \lambda_l\geq 1 \mbox{ such that for  any }  \ell \in L(x), \ \ 0\leq \ell \leq c_l(1+|x|^{\lambda_l});  
\end{cases}
\]
For $x \in \R^d$ and $p\in \R^d$, we define:
$$
H(x,p)=\sup_{a \in \As }\{-p \cdot f(x,a)-\ell(x,a)\}.
$$
Let  $T>0$ be given final time, we consider for each $i=1,\cdots, m$ the following  set of HJB equations: 
\begin{equation}\label{HJBOmi}
 -\partial_t u(t,x)+H(x,Du(t,x))=0  \text{ for }\, t\in (0,T),\ x\in\Omega_i,
 \end{equation}
combined with the final condition 
$$
 u(T,x)=\vp(x)  \text{ for }\ x\in\Omega_i.
$$
The system above implies that on each domain $\Omega_i$ a classical HJ equation is considered.
However, there is no information on the boundaries of the domains which are the junctions
between $\Omega_i$ We then address the question to know what condition should be considered on the
boundaries in order to get the existence and uniqueness of solution to all the equations.\\
Here $\vp$ is called the final cost function and two different assumptions on $\vp$ are considered in this work:
\begin{description}
 \item[(H$\vp$1)]
 $\vp$ is a  Lipschitz continuous function,
\end{description}
\begin{description}
 \item[(H$\vp$2)]
 $\vp$ is a  lower semi-continuous function with $\lambda_{\vp}$-superlinear growth for some $\lambda_{\vp}\geq 1$.
\end{description} 
For the rest of the paper we set
\begin{equation}\label{lambda}
\lambda=\max\{\lambda_l,\lambda_{\vp}\}.
\end{equation}

We consider the HJB equation \eqref{HJBOmi} in each subdomain $\Omega_i$  and we then address the question to know which are the junction conditions on the interface $\Gamma$ to get the existence and uniqueness of solution to \eqref{HJBOmi}.

\bigskip

A technical efficient way to deal with the running cost is to introduce an augmented dynamics.  To this end we define 
$$
b(x,a)=c_l(1+|x|^{\lambda_l})-\ell(x,a) \quad \mbox{ for any } x \in \mathbb{R}^d, a \in \As . 
$$
For each $x\in\R^d$, we define the augmented  dynamics $G :\R^d\rightsquigarrow\R^d$
\[
G(x):=\{(f(x,a),-\ell(x,a)-r)\,:\, a \in \As ,\ 0\leq r\leq b(x,a)\}.
\]

It is not difficult to see by ${\bf (HF),\bf (HL)}$ that  this map has non empty compact images.
Moreover, we also suppose the following assumption.
\begin{enumerate}
\item[${\bf(HG)} $] $\textbf{G}(\cdot)$ has convex images.
\end{enumerate}

\subsection{Tangential and Essential dynamics. Controllability assumptions}
An important type of dynamics is the notion of {\em tangent dynamics} considered as the intersection of the convexified dynamics $F$ and the tangent space to each subdomain. We first recall the notion of tangent cone. For any $\mathcal{C} \subset \R^p$ with $1\leq p\leq d$, the tangent cone $\T_{\mathcal{C}}(x)$ at $x \in \mathcal{C}$ is defined by
$$
\T_{\mathcal{C}}(x)=\{v \in \R^p \, : \, \liminf_{t \to 0^{+}} \frac{d_\C(x+tv)}{t}=0\},
$$
where $d_{\mathcal{C}}(\cdot)$ is the distance function to $\mathcal{C}$. Note that $\T_{\Gamma_j}(x)$ agrees with the tangent space of $\Gamma_j$ at $x$ for $j=1,\cdots, l$ and the dimension of $\T_{\Gamma_j}$ is strictly smaller than $d$.\\
On each $\M_k$, the set of tangent dynamics is a multifunction $F_{\M_k}:\overline{\M_k}\rightsquigarrow\R^d$ defined as
\[
F_{\M_k}(x)=F(x)\cap\T_{\M_k}(x),\ \forall\,x\in\overline{\M_k}.
\]
Here $\T_{\M_k}(x)$ agrees with the tangent space of $\M_k$ at $x$ with the same dimension of $\M_k$, which can be extended up to $\overline{\M_k}$ by continuity. 

Correspondingly the set of controls $A_{\M_k}$ related to the tangent dynamics on each $\M_k$ is set by
\[
A_{\M_k}(x)=\{a\in \As \,:\,f(x,a)\in \T_{\M_k}(x)\},\ \forall\,x\in\M_k.
\]

\smallskip

The next notion of dynamics is the {\em essential dynamics} $F^E$ firstly introduced in \cite{BW}, and the definition is given as follows. 
\begin{dfntn}\label{DefFE}
For any $x\in\R^d$, the multifunction $F^E:\R^d\rightsquigarrow\R^d$ at $x$ is defined by
\[
F^E(x):=\bigcup\,\{F^E_{\M_k}(x)\,:\,x\in\overline \M_k,\ k\in \{1,\cdots, l+m\}\},
\]
where $F^E_{\M_k}:\overline{\M_k}\rightsquigarrow\R^d$ is defined by
\[
F^E_{\M_k}(x)=F^{\mbox{ext}}_k(x)\cap\T_{\overline{\M_k}}(x),\ \text{for}\ x\in\overline{\M_k},
\]
where $F^{\mbox{ext}}_k: \overline{\mathcal{M}_k} \rightsquigarrow \R^d$ is the extension by continuity of $F|_{\mathcal{M}_k}$ to $\overline{\mathcal{M}_k}$.\\
We define also the set of controls corresponding to the essential multifunction: $\forall\,x\in\R^d$,
\[
A^E_{\M_k}(x):=\{a\in \As\,:\,f(x,a)\in F^E_{\M_k}(x)\},\ A^E(x):=\cup \left\{A^E_{\M_k}(x)\,:\,x\in\overline{\M_k},\ k\in \{1,\cdots, l+m\}\right\}.
\]
\end{dfntn}
We define also the essential dynamics for the augmented dynamics as follows.
\begin{dfntn}\label{DEFG}
For each $x\in\R^d$, we define the augmented essential dynamics 
\[
G^E(x):=\{(f(x,a),-\ell(x,a)-r)\,:\,  0\leq r\leq b(x,a), a \in A^E(x)\}.
\]
For each $\M_k$, , the augmented tangent dynamics is the following
\[
G_{\M_k}(x):=\{(f(x,a),-\ell(x,a)-r)\,:\,f(x,a)\in\mathcal{T}_{\M_k}(x),\ 0\leq r\leq b(x,a),\ \ a\in \As\}.
\]
\end{dfntn}

To state the main results, we shall need also some controllability assumptions around the interfaces. Since two cases cases will be studied where either {\bf (H$\vp$1)} or {\bf (H$\vp$2)} is satisfied,  different hypotheses of controllability are required in each case. Combined with {\bf (H$\vp$1)}, the following controllability condition will be assumed.

\begin{description}
 \item[(H2)]
  There exists $r_1>0$ such that for any $x\in\Gamma_j$
 \[
 B(0,r_1) \subset F(x).
 \]
\end{description}

Under the assumption {\bf (H$\vp$2)}, we shall consider the following weaker hypothesis:
\begin{description}
 \item[(H3)] For each $j=0,\ldots,l$, one of the following properties is satisfied on $\Gamma_j$.
 \begin{itemize}
  \item Either  any $x\in\Gamma_j$, 
  \[
  F(x)\cap \mathcal T_{\Gamma_j}(x)=\emptyset;
  \]
  \item Or there exists $r_2>0$ such that for any $x\in\Gamma_j$,
  \[
  B(0,r_2)\subset F(x).
  \]
 \end{itemize}
\end{description}

Let us point out that {\bf (H3)} is a much weaker assumption than {\bf (H2)}. 
Indeed, consider the simple case of two domains in $\R$ with $\Om_1=\{x\,:\,x<0\}$, $\Om_2=\{x\,:\,x>0\}$ and $\Gamma=\{0\}$. For any $x\in\R$, let
\[
F(x)=\{1\}.
\]
In this case where $F$ is Lipschitz continuous everywhere, on the interface $\Gamma$ we have
\[
F(0)\cap \mathcal T_{\Gamma}(0)=\emptyset.
\]
Thus, {\bf (H3)} is satisfied while {\bf (H2)} is not obeyed.

Note that the controllability assumptions  {\bf (H2)} and {\bf (H3)} imply different properties on the tangential dynamics. Indeed, we have the following results  whose proofs are postponed to the Appendix A.
\begin{prpstn}\label{PROPGM}
 Assume {\bf (H1)}, {\bf (HF)}, {\bf (HL)},  {\bf (H2)}. Then $G_{\Gamma_j}$ is locally Lipschitz continuous on $\Gamma_j$.
\end{prpstn}

\begin{prpstn}\label{PROPH3}
 Assume {\bf (H1)}, {\bf (HF)}, {\bf (HL)},  {\bf (H3)}. Then the following holds. 
 \begin{enumerate}[(i)]
  \item $G_{\Gamma_j}$ is either with empty images or locally Lipschitz continuous on $\Gamma_j$.
  \item For each $j=0,\ldots,l$ and $x\in\Gamma_j$ with $F_{\Gamma_j}(x)\neq\emptyset$, there exists $\e_j,\Delta_j>0$ such that
 \[
 \mathcal R(x;t)\cap\overline{\Gamma_j}\subseteq \bigcup_{s\in[0,\Delta_j t]}\mathcal{R}_j(x;s),\ \forall\,t\in[0,\e_j],
 \]
 where
 \[
 \mathcal R(x;t):=\{y(t)\,:\,\dot y(s)\in F(y(s))\ \text{a.e.}\ s\in (0,t),\ y(0)=x\},
 \]
 \[
 \mathcal R_j(x;t):=\{y(t)\,:\,\dot y(s)\in F_{\Gamma_j}(y(s))\ \text{a.e.}\ s\in (0,t),\ y(0)=x\}.
 \]
 \end{enumerate}
\end{prpstn}
\subsection{Main results}
We define  the following  Hamiltonians: $H_F,H^E:\R^d\times\R^d\ra\R$ and $H_{\Gamma_j}:\Gamma_j\times\R^d\ra\R$  
\[
 H_F(x,p)=\sup_{a\in \As}\{-p\cdot f(x,a)-\ell(x,a)\},
\]
\[
 H^E(x,p)=\sup_{a\in A^E(x)}\{-p\cdot f(x,a)-\ell(x,a)\},
\]
and
\[
H_{\Gamma_j}(x,p)=\sup_{a\in A_{\Gamma_j}(x)}\{-p\cdot f(x,a)-\ell(x,a)\}.
\]
We consider the following two kind of junction conditions:
\begin{equation}\label{HJBHE}
 -\partial_t u(t,x)+H^E(x,Du(t,x))=0,\ \text{for}\ t\in(0,T),\ x\in\Gamma_j,\ 
\end{equation}
\begin{equation}\label{HJBHGamma}
 \left\{
 \begin{array}{ll}
  -\partial_t u(t,x)+H_F(x,Du(t,x))\geq0,\ \text{for}\ t\in(0,T),\ x\in\Gamma_j,\ \\
  -\partial_t u(t,x)+H_{\Gamma_j}(x,Du(t,x))\leq 0,\ \text{for}\ t\in(0,T),\ x\in\Gamma_j.
 \end{array}
 \right.
\end{equation}
The viscosity sense of the solutions to the above equations/inequalities needs to be clarified. 
Before giving the definition of solutions, we recall the notion of extended differentials.\\
Let $\phi:(0,T)\times\R^d\ra\R$ be a continuous function, and let $\M\subseteq\R^d$ be an open $C^2$ embedded manifold in $\R^d$. 
Suppose that $\phi\in C^1((0,T)\times \M)$, we define the differential of $\phi$ on any $(t,x)\in(0,T)\times\overline{\M}$ by
\[
\nabla_{\overline{\M}}\phi(t,x):=\lim_{x_n\ra x,x_n\in\M}\left(\phi_t(t,x_n),D\phi(t,x_n)\right).
\]
Note that $\nabla\phi$ is continuous on $(0,T)\times\M$, the differential defined above is actually the extension of $\nabla\phi$ to the whole $\overline{\M}$.
\smallskip

The precise viscosity and bilateral viscosity notions are given as follows.  
\begin{dfntn}\label{DEFBvsup}{\bf (Viscosity supersolution)}\\
Let $u:(0,T]\times\R^d\ra\R$. We say that 
$u$ is a supersolution of \eqref{HJBOmi}-\eqref{HJBHE} (\eqref{HJBOmi}-\eqref{HJBHGamma} resp.) if $u$ is lsc and
for any $(t_0,x_0)\in(0,T)\times\R^d$ and $\phi\in C^1((0,T)\times\R^d)$ such that $u-\phi$ attains a local minimum at $(t_0,x_0)$, we have
\[
-\phi_t(t_0,x_0)+H^E(x_0,D\phi(t_0,x_0))\geq 0
\]
\[
\text{(}-\phi_t(t_0,x_0)+H_F(x_0,D\phi(t_0,x_0))\geq 0,\ \text{resp.)}.
\]
\end{dfntn}
\begin{dfntn}\label{DEFBvsub}{\bf (Viscosity subsolution)}\\
Let $u:(0,T]\times\R^d\ra\R$. 
\begin{enumerate}
\item 
$u$ is a subsolution of \eqref{HJBOmi}-\eqref{HJBHE} if $u$ is usc and
for any $(t_0,x_0)\in(0,T)\times\R^d$, any $k\in\{0,\ldots,m+l\}$ with $x_0\in\overline{\M_k}$ and any continuous $\phi:(0,T)\times\R^d\ra\R$ with $\phi|_{(0,T)\times\overline{\M_k}}$ being $C^1$ such that $u-\phi$ attains a local maximum at $(t_0,x_0)$ on $(0,T)\times\overline{\M_k}$, we have
\[
-p_t+\sup_{a\in A^E_{\M_k}(x_0)}\{-p_x\cdot f(x_0,a)-\ell(x_0,a)\}\leq 0,\ \text{with}\ (p_t,p_x)=\nabla_{\overline{\M_k}}\phi(t_0,x_0).
\]
\item
$u$ is a subsolution of \eqref{HJBOmi}-\eqref{HJBHGamma} if $u$ is usc and
for any $(t_0,x_0)\in(0,T)\times\R^d$, any $k\in\{0,\ldots,m+l\}$ with $x_0\in\M_k$ and any $\phi\in C^1((0,T)\times\M_k)$ such that $u|_{\M_k}-\phi$ attains a local maximum at $(t_0,x_0)$, we have
\[
-\phi_t(t_0,x_0)+H_{\M_k}(x_0,D\phi(t_0,x_0))\leq 0.
\]
\item
$u$ is a bilateral subsolution of \eqref{HJBOmi}-\eqref{HJBHGamma} if $u$ is lsc and
for any $(t_0,x_0)\in(0,T)\times\R^d$, any $k\in\{0,\ldots,m+l\}$ with $x_0\in\M_k$ and any $\phi\in C^1((0,T)\times\M_k)$ such that $u|_{\M_k}-\phi$ attains a local minimum at $(t_0,x_0)$, we have
\[
-\phi_t(t_0,x_0)+H_{\M_k}(x_0,D\phi(t_0,x_0))\leq 0.
\]
\end{enumerate}
\end{dfntn}
\begin{dfntn}\label{DEFBvs}{\bf (Viscosity solution and bilateral viscosity solution) }
\begin{enumerate}
\item
$u$ is a viscosity solution to \eqref{HJBOmi}-\eqref{HJBHE} (\eqref{HJBOmi}-\eqref{HJBHGamma} resp.) if $u$ is both a supersolution in the sense of Definition \ref{DEFBvsup} and a subsolution of \eqref{HJBOmi}-\eqref{HJBHE} (\eqref{HJBOmi}-\eqref{HJBHGamma} resp.) in the sense of Definition \ref{DEFBvsub}, and $u$ satisfies the final condition
\[
u(T,x)=\vp(x),\ \forall\,x\in\R^d.
\]
\item
$u$ is a bilateral viscosity solution to \eqref{HJBOmi}-\eqref{HJBHGamma} if $u$ is both a supersolution in the  sense of Definition \ref{DEFBvsup} and a bilateral subsolution of \eqref{HJBOmi}-\eqref{HJBHGamma} in the sense of Definition \ref{DEFBvsub}, and $u$ satisfies the final condition
\[
u(T,x)=\vp(x),\ \forall\,x\in\R^d.
\]
\end{enumerate}
\end{dfntn}
The main results are the following two theorems under  assumption {\bf (H$\vp$1)} and  {\bf (H$\vp$2)} respectively.
\begin{thrm}\label{THMHEHGamma}
 Assume {\bf (H$\vp$1)}, {\bf (H1)}, {\bf (HF)}, {\bf (HL)}, {  \bf (HG)}, {\bf(H2) }. The systems \eqref{HJBOmi}-\eqref{HJBHE} and \eqref{HJBOmi}-\eqref{HJBHGamma} have the same unique continuous viscosity solution  (in the sense of Definition \ref{DEFBvs}) with restriction on $[0,T]\times \Gamma$ locally Lipschitz continuous and  with $\lambda_l$-superlinear growth.
\end{thrm}

\begin{thrm}\label{THMBilateral}
 Assume {\bf (H$\vp$2)},{\bf (H1)}, {\bf (HF)}, {\bf (HL)}, {\bf (HG)}, {\bf(H3)}. The system \eqref{HJBOmi}-\eqref{HJBHGamma} has a unique lsc bilateral viscosity solution with $\lambda$-superlinear growth (in the sense of Definition \ref{DEFBvs}).
\end{thrm}

The key issues in the framework of multi-domains involve the controllability assumptions on the interfaces and the continuity of the solutions of HJB equations. Our first contribution is the existence and uniqueness result in the class of continuous solutions under the assumption {\bf (H2)} that the controllability holds everywhere on the interfaces. Similar results in this case can be found in the literature in \cite{BBC,RZ,BBC2,RSZ,BC} with different settings of multi-domains and transmission conditions on the interfaces. The second contribution is the existence and uniqueness result in the class of discontinuous solutions and the controllability condition can be weakened on the interfaces. This is new in the literature and a similar situation is discussed in \cite{HZ} in the state constrained case. \\
Finally we mention that Section \ref{sres} is devoted to a stability result under the hypothesis \textbf{(H$\vp$1)} when approaching
$\varphi$ by a sequence of Lipschitz continuous functions.

\section{Main properties of the  value function}\label{propvf}
Consider the value function associated to the  control problem on $\R^d$ defined, for any $(t,x)\in [0,T]\times\R^d$, as
\begin{equation}\label{oc}
v(t,x):=\inf\left\{\vp(y(T))+\int^T_t\ell(y(s),\alpha(s))ds\,:\,(y(\cdot),\alpha(\cdot))\ \text{satisfies \eqref{DSy}}\right\},
\end{equation}
 where $\alpha \in \A:=L^{\infty}(0,T;\As)$  and $(y,\alpha)$ satisfy
\begin{equation}\label{DSy}
 \left\{
 \begin{array}{ll}
  \dot y(s)= f(y(s),\alpha(s)) & \text{a.e.}\ s\in(t,T).\\
  y(t)=x.
 \end{array}
 \right.
\end{equation}

For any $(x,t) \in \R^d \times [0,T]$, we denote by $\mathcal{S}_t^T(x)$ the set of trajectories $y(\cdot)$ satisfying \eqref{DSy}.
We remark that the optimal control problem \eqref{oc} can be written in terms of the convex augmented dynamic $G$ as follows. For any $t\in[0,T]$, $x\in\R^d$, consider the differential inclusion: 
\begin{equation}\label{DIyeta}
  (\dot y(s),\dot \eta(s))\in G(y(s)) \quad s\in(t,T),
  \end{equation}
then the control problem \eqref{oc} is equivalent to:
\begin{equation}\label{OCG}
 \vartheta(t,x):=\inf\{\vp(y(T))-\eta(T)\,:\,(y(\cdot),\eta(\cdot))\ \text{satisfies \eqref{DIyeta} with }
  (y(t),\eta(t))=(x,0)\}.
\end{equation}
Note  that $G$ is upper semi-continuous with compact and convex images and therefore by standard arguments one can prove that  \eqref{OCG} admits a solution.

\begin{rmrk}\label{grownrem}
Let $(y(\cdot),\eta(\cdot))$ satisfy \eqref{DIyeta} with $y(t)=x,\eta(t)=0$.
The Gronwall lemma implies
$$
|y(s)|\leq (1+|x|)e^{c_f(t-s)} \quad \forall s \in (t,T)
$$
and
$$
|\dot{y}(s)|\leq c_f(1+|x|)e^{c_f(t-s)} \quad \forall s \in (t,T).
$$
Moreover, since $\lambda_l \geq 1$
 \begin{equation}\label{quadr}
 l(y(s), \alpha(s))\leq c_l(1+|y(s)|)^{\lambda_l}e^{\lambda_lc_f(t-s)} \quad \forall s \in (t,T),
 \end{equation}
and therefore
$$
|\dot{\eta}(s)|\leq c_l(1+|y(s)|)^{\lambda_l}e^{\lambda_lc_f(t-s)} \quad \forall s \in (t,T).
$$
\end{rmrk}

We recall the principal properties of the value function.
\begin{prpstn}\label{lambdagr}
 Assume {\bf(HF)}, {\bf (HL)}. Let $\lambda_l \geq 1$ and $\lambda \geq 1$ be defined respectively as in \textbf{HL} and \eqref{lambda}. Under the assumption {\bf(H$\vp$1)} (or {\bf(H$\vp$2)} resp.), the value function $v(t,\cdot)$ has $\lambda_l$ (or $\lambda$ resp.) superlinear growth on $\R^d$.
\end{prpstn}
\begin{proof}
By {\bf(H$\vp$1)} the final cost $\vp$ is Lipschitz continuous and then it has linear growth. Then, the proof follows from \eqref{quadr} of Remark \ref{grownrem} and the linear  growth of $\vp$. Similarly, under {\bf(H$\vp$2)} we get the $\lambda$-superlinear growth of $v$.
\end{proof}
\subsection{The Dynamic Programming Principle}
A well-known and key result is that the value function $v$ satisfies a {\em Dynamical Programming Principle (DPP)}.
\begin{prpstn}\label{PROdpp}
 Assume {\bf (H1)}, {\bf (HF)}(i), {\bf (HL)}(i). For any $(t,x)\in [0,T]\times\R^d$, the following holds.
 \begin{enumerate}[(i)]
  \item $v$ satisfies the {\em super-optimality}, i.e.  there exists $(\bar y,\bar \alpha)$ satisfying \eqref{DSy} such that
  \[
  v(t,x)\geq v(t+h,\bar y(t+h)) +\int^{t+h}_t \ell(\bar y(s),\bar \alpha(s))ds,\ \text{for}\ h\in[0,T-t].
  \]
  \item $v$ satisfies the {\em sub-optimality}, i.e. for any $(y,\alpha)$ satisfying \eqref{DSy}
  \[
  v(t,x)\leq v(t+h,y(t+h)) +\int^{t+h}_t \ell( y(s),\alpha(s))ds,\ \text{for}\ h\in[0,T-t].
  \]
 \end{enumerate}
\end{prpstn}
 
In the following proposition, we state a backward sub-optimality for the value function. Note that the proof follows  by standard arguments as a consequence of Proposition \ref{PROdpp}. \\
In this case we look at the following system. 
\begin{prpstn}\label{DPPback}
Assume {\bf(H1)}, {\bf (HF)}(i), {\bf (HL)}(i). For any $(t,x)\in[0,T]\times\R^d$, $y(\cdot),\alpha$ satisfying  
\begin{equation}\label{DSyback}
 \left\{
 \begin{array}{ll}
  \dot y(s)= f(y(s),\alpha(s)) & \text{a.e.}\ s\in(0,t),\\
  y(t)=x,
 \end{array}
 \right.
\end{equation}
it holds
 \[
 v(t,x)\geq v(t-h,y(t-h))-\int^t_{t-h}\ell(y(s),\alpha(s))ds,\ \forall\,h\in[0,t].
\]
\end{prpstn}

Now we recall the properties satisfied by the value function in the two cases studied, precisely the continuity under the assumption {\bf (H$\vp$1)} and the lower semi-continuity under the assumption {\bf (H$\vp$2)}.

\subsection{Lower semicontinuity under (H$\vp$2)}

Under {\bf(H$\vp$2)}  $v$ is lower semi-continuous. In this case we characterize the value function $v$ through the backward sub-optimality, as showed in the following proposition.
 
\begin{prpstn}
 Assume {\bf(H$\vp$2)}, {\bf(H1)}, {\bf (HF)}(i), {\bf (HL)}(i), {\bf(HG)}. Then $v$ is lower semi-continuous. Moreover, for any $(t,x)\in[0,T]\times\R^d$ and $y(\cdot)$ satisfying \eqref{DSyback},
  \begin{equation}\label{propv}
  v(t,x)=\lim_{h\ra 0^+}v(t-h,y(t-h)).
  \end{equation}

\end{prpstn}
\begin{proof}
The lower semi-continuity essentially follows from the upper semi-continuity, convexity and compactness of the dynamics $G$ and from the lower semi-continuity of $\vp$. Since it is a standard result we omit the details  of the proof. We prove  \eqref{propv}.
 By the lower semi-continuity of $v$ we have
 \begin{equation}\label{ast}
 v(t,x)\leq \mathop{\lim\,\inf}\limits_{h\ra 0^+} v(t-h,y(t-h)).
 \end{equation}
By Proposition \eqref{DPPback}, we get
 \[
 v(t,x)\geq v(t-h,y(t-h))+\int^t_{t-h}\ell(y(s),\alpha(s))ds,\ \forall\,h\in[0,t],
 \]
 and then we have
 \begin{equation}\label{pal}
 v(t,x)\geq \mathop{\lim\,\sup}\limits_{h\ra 0^+} v(t-h,y(t-h)).
 \end{equation}
 By \eqref{ast} and \eqref{pal} we conclude that $v(t,x)=\lim_{h\ra 0^+}v(t-h,y(t-h))$.

\end{proof}
\subsection{Continuity under (H$\vp$1)}

Under {\bf(H$\vp$1)} and the controllability assumption {\bf (H2)}, we have the continuity of the value function.

\begin{prpstn}\label{contv}
Assume {\bf(H$\vp$1)}, {\bf(H1)}, {\bf (HF)}, {\bf (HL)}, {\bf (HG)}, {\bf(H2)}. Then $v$ is continuous on $[0,T] \times \R^d$. Moreover, $v|_{[0,T]\times\Gamma}$ is locally Lipschitz continuous on $[0,T]\times\Gamma$.
\end{prpstn}

The proof  is divided in three steps. First we prove the local Lipschitz continuity of the space restriction of $v$ on $\Gamma$, then the continuity of $v$ on $\Gamma$ is obtained and finally the continuity of $v$ in $\R^d$ is concluded. The proof is inspired by the arguments used in \cite{RSZ, RZ} and is given  in  Appendix A. However, we remark that in \cite{RZ}  a total controllability is assumed in each subdomains (and not only on the interfaces as in \textbf{H2}), which leads to the Lipschitz continuity of the value function in all the space.

\begin{rmrk}\rm{
 We remark that our results can be proved under the  following weaker controllability assumption, which divides \textbf{(H2)} into the tangential controllability assumption \textbf{(P1)} and the normal one \textbf{(P2)}:
\begin{description}
 \item[(P1)]
  There exists $r_1>0$ such that for any $x\in\Gamma_j$, 
 \[
 B(0,r_1)\cap \T_{\Gamma_j}(x)\subset F(x).
 \]
 \item[(P2)]
  There exists $r_2>0$ such that for any $x\in\Gamma_j$
 \[
 B(0,r_2)\cap \T_{\Gamma_j}(x)^{\bot}\subset F(x).
 \]
\end{description}
The normal controllability of \textbf{(P2)} is needed to have the local Lipschitz regularity of the augmented dynamics $G_{\Gamma_j}$ (see Proposition \ref{PROPGM}). The tangential controllability stated in \textbf{(P1)} is used to prove the local Lipschitz regularity of the restriction of the value function  on $[0,T] \times \Gamma$ (see Proposition \ref{contv}).   However, we mention that \textbf{(P1)} is not necessary in order to have the Lipschitz regularity. Indeed, 
consider the case of two-domains in $\R^2$ with
\[
\Om_1=\{(x_1,x_2)\,:\, x_1<0,\ x_2\in\R\},\ \Om_2=\{(x_1,x_2)\,:\, x_1>0,\ x_2\in\R\}.
\]
and the interface
\[
\Gamma=\{(0,x_2)\,:\,x_2\in\R\}.
\]
Suppose that the dynamics is defined as follows:
\[
F(x)=\left\{
\begin{array}{lll}
\{(-1,0)\} & \text{for}\ x\in\Om_1,\\
\{(1,0)\} & \text{for}\ x\in\Om_2,\\
\{(p,0)\,:\,p\in[-1,1]\} & \text{for}\ x\in\Gamma.
\end{array}
\right.
\]
The cost functions are the following: for $x=(x_1,x_2)\in\R^2$

\[
\vp(x)=|x_1|,\ \ell\equiv 0.
\]
Note that for $x\in\Gamma$ and any $r_1>0$
\[
B(0,r_1)\cap\mathcal T_{\Gamma}(x)=\{(0,p)\,:\,p\in [-r_1,r_1]\}
\]
which is not included in $F(x)$. Therefore, {\bf (P1)} is not satisfied in this case.

Now we compute the value function, we refer to \eqref{oc} for the definition. We have for $(t,x)\in (0,T)\times\R^2$
\[
y_{t,x}(T)\left\{
\begin{array}{lll}
=(x_1-T+t,x_2) & \text{for}\ x\in\Om_1,\\
=(x_1+T-t,x_2) & \text{for}\ x\in\Om_2,\\
\in\{(t-T,x_2),(T-t,x_2)\} & \text{for}\ x\in\Gamma.
\end{array}
\right.
\]
It is then deduced that
\[
v(t,x)=T-t+|x_1|,\ \forall\,x=(x_1,x_2)\in\R^2,
\]
which is globally Lipschitz continuous. Therefore, the tangential controllability condition {\bf (P1)} is not a necessary condition for the local Lipschitz continuity of the restriction of the value function on $[0,T]\times\Gamma$.

}\end{rmrk}

\section{Supersolutions and super-optimality}\label{supsup}
This section is devoted to the characterization of the super-optimality via HJB inequalities.
The characterization through the tangential dynamic is a classical result since $F$ is upper semi-continuous and $G$ is convex. We give also a  more precise characterization through the essential dynamics, which is not standard since in general $F_E$ is not usc. The proof is mainly based on the fact that the set of trajectories driven by $F$ and $F_E$ are the same. We refer to \cite[Proposition 3.4]{RZ} for a proof of this result. Finally, we remark that in the following theorem no controllability assumption is needed.

\smallskip

The characterization of the super-optimality is the following.
\begin{thrm}\label{CHARAsuper}
 Assume {\bf(H1)}, {\bf (HF)}, {\bf (HL)}, {\bf(HG)}. Let $u:[0,T]\times\R^d\ra\R$ be a  lsc function. The following are equivalent.
 \begin{enumerate}[(i)]
  \item[(i)]
  $u$ satisfies the super-optimality;
  \item[(ii)]
  $u$ is a supersolution to \eqref{HJBOmi}-\eqref{HJBHE};
  \item[(iii)]
  $u$ is a supersolution to \eqref{HJBOmi}-\eqref{HJBHGamma}.
 \end{enumerate}
\end{thrm}
\begin{proof}
The implication (iii) $\Rightarrow$ (i)  is customary and well known, in particular see \cite{HZ}, Proposition 5.1 in the constrained framework and \cite{F}, \cite{FP}, \cite{CLSW}, \cite{WZ} for the unconstrained framework.\\
Now we prove that (i)$\Rightarrow$ (ii).
 Given $t\in[0,T]$ and $x\in\R^d$, by the super-optimality of $u$ there exists $\bar y,\ \bar \alpha$ such that
 \[
 u(t,x)\geq u(t+h,\bar y(t+h))+\int^{t+h}_t \ell(\bar y(s),\bar \alpha(s))ds,\ \forall\,h\in [0,T-t].
 \]
 We set
 \[
 \bar \eta(h):=u(t,x)-\int^{t+h}_t\ell(\bar y(s),\bar \alpha(s))ds.
 \]
 For any $\phi\in C^1((0,T)\times\R^d)$ such that $u-\phi$ attains a local minimum at $(t,x)$, we have
 \[
 \bar \eta(h)\geq u(t+h,\bar y(t+h)) \geq \phi(t+h,\bar y(t+h))+u(t,x)-\phi(t,x),\ \forall\,h\in [0,T-t],
 \]
 i.e.
 \[
 \phi(t,x)-\phi(t+h,\bar y(t+h))+\bar \eta(h)-\bar \eta(0)\geq 0,\ \forall\,h\in [0,T-t].
 \]
 Up to a subsequence, let $h_n\ra 0^+$ such that there exists $\bar p\in\R^d,\ \bar q\in\R$ satisfying
 \[
 \frac{\bar y(t+h_n)-x}{h_n}\ra \bar p,\ \frac{\bar \eta(h_n)-\bar \eta(0)}{h_n}\ra \bar q.
 \]
 It is clear that $(\bar p,\bar q)\in G(x)$ since $G$ is usc and convex valued. Moreover, by \cite[Lemma 3.6]{RZ}
 \[
 \bar p\in \co F^E(x).
 \]
 Therefore, by the definition of $G^E$
 \[
 (\bar p,\bar q)\in \co G^E(x).
 \]
 We then deduce that
 \[
 -\phi_t(t,x)+\sup_{(p,q)\in \co G^E(x)}\{-p\cdot D\phi(t,x)+q\}\geq 0.
 \]
 The separation theorem implies that
 \[
 -\phi_t(t,x)+\sup_{(p,q)\in G^E(x)}\{-p\cdot D\phi(t,x)+q\}\geq 0.
 \]
 By the definition of $G^E(x)$, for any $(p,q)\in G^E(x)$, there exists $a\in A^E(x)$ such that
 \[
 p= f(x,a),\ q\leq -\ell(x,a).
 \]
 Thus, we conclude that
 \[
 -\phi_t(t,x)+\sup_{a\in A^E(x)}\{-f(x,a)\cdot D\phi(t,x)-\ell(x,a)\}\geq 0,
 \]
which ends the proof.\\
Now we prove that (ii) $\Rightarrow$ (iii).
Let $u$ be a supersolution to \eqref{HJBOmi}-\eqref{HJBHE}, for any $(t,x)\in(0,T)\times\R^d$ and $\phi\in C^1((0,T)\times\R^d)$ such that $u-\phi$ attains a local minimum at $(t,x)$, we have
\[
-\partial_t \phi(t,x)+\sup_{a\in A^E(x)}\{-f(x,a)\cdot D\phi(t,x)-\ell(x,a)\}\geq 0.
\]
Note that $A^E(x)\subset A(x)$, then
\[
-\partial_t \phi(t,x)+\sup_{a\in A(x)}\{-f(x,a)\cdot D\phi(t,x)-\ell(x,a)\}\geq 0,
\]
which is the desired result.
\end{proof}

\section{Subsolutions and sub-optimality}\label{subsub}
This section is devoted to the characterization of the sub-optimality via the HJB inequalities.
In the standard setting where the dynamics are not stratified, the multifunction of dynamics has to be Lipschitz to obtain the characterization of the sub-optimality. This property is not satisfied in our case and no classical arguments can be adapted here. However, we mention that on each subdomain the (augmented) dynamics are locally Lipschitz continuous as indicated in Proposition \ref{PROPGM}. Here is to investigate the desired sub-optimality property in each subdomain, and then the properties are glued together to obtain the complete characterization result. This idea was firstly introduced in \cite{BW}.

\smallskip

The characterization of the sub-optimality is the following. We split it into two theorems depending whether we assume {\bf(H2)} (Theorem \ref{CHARAsub}) or {\bf(H3)} (Theorem \ref{CHARAsubback}). 
\begin{thrm}\label{CHARAsubback}
 Assume {\bf(H1)}, {\bf (HF)}, {\bf (HL)},  {\bf(HG)},  {\bf(H3)}. Let $u:[0,T]\times\R^d\ra\R$ be a lsc function. Then the following are equivalent.
 \begin{enumerate}[(i)]
  \item 
  $u$ satisfies the sub-optimality;
  \item
  $u$ is the bilateral subsolution to \eqref{HJBOmi}-\eqref{HJBHGamma}.
 \end{enumerate}
\end{thrm}
Since the proof of Theorem \ref{CHARAsubback} follows  the strategy used for a stratified state constrained Mayer problem in  \cite{HZW}, Proposition $3.5$, we give it in Appendix $B$.  However, we remark that our setting is different from \cite{HZW}, in particular we have discontinuous and unbounded dynamic and cost on each interfaces.

\begin{thrm}\label{CHARAsub}
 Assume {\bf(H1)}, {\bf (HF)}, {\bf (HL)}, {\bf(HG)}, { \bf(H2)}. Let $u:[0,T]\times\R^d\ra\R$ be an usc function, such that $u$ is continuous on $\Gamma$ and the restriction of $u$ on $[0,T]\times \Gamma$ is locally Lipschitz continuous. Then the following are equivalent.
 \begin{enumerate}[(i)]
  \item 
  $u$ satisfies the sub-optimality;
  \item
  $u$ is the subsolution to \eqref{HJBOmi}-\eqref{HJBHE};
  \item
  $u$ is the subsolution to \eqref{HJBOmi}-\eqref{HJBHGamma}.
 \end{enumerate}
\end{thrm}

\begin{proof}

Note that  (ii) $\Rightarrow$ (iii) follows since, for any $x\in\M_k$ with $k\in\{0,\ldots,l+m\}$, every element of $A^E_{\M_k}(x)$ belongs to $A^E(x)$ and then $H_{\M_k}(x,\cdot)\leq H^E(x,\cdot)$.\\
Now we prove that (i)$\Rightarrow $ (ii).
First we remark that the significant role of the essential dynamics $F^E$ is that any dynamic in $F^E$ is used by some trajectories as stated in the following lemma. For the proof we refer to \cite[Lemma 3.9]{RZ}.

\begin{lmm}\label{corGE}
Assume {\bf(H1)}, {\bf (HF)}, {\bf (HL)}, {\bf(HG)}.  Let $k\in\{0,\ldots,m+l\}$, $t\in[0,T)$ and $x\in\overline{\M_k}$. Then for any $(p,q)\in G^E_{\M_k}(x)$, $\xi\in\R$,
 there exist $\tau>t$ and a $C^1$ trajectory $(y(\cdot),\eta(\cdot))$ satisfying \eqref{DIyeta} in
$ (t,\tau)$ such that $y(t)=x, \eta(t)=\xi$,
 with $(\dot y(t),\dot \eta(t))=(p,q)$ and $y(s)\in\overline{\M_k}$ for $s\in [t,\tau]$. 
\end{lmm}

Then Lemma \ref{corGE} implies that, for any $k\in\{0,\ldots,m+l\}$, $t\in (0,T)$, $x\in\R^d$ and $a\in A^E(x)$ such that $f(x,a)\in F^E_{\M_k}(x)$ where $x\in\overline{\M_k}$,   there exists $\tau>t$, $y,\eta\in C^1[t,\tau)$ satisfying \eqref{DIyeta} in $ (t,\tau)$ such that $ y(t)=x,\ \eta(t)=u(t,x)$
with $(\dot y(t),\dot \eta(t))=(f(x,a),-\ell(x,a))$ and $y(s)\in\overline{\M_k}$ for $s\in [t,\tau]$.

The sub-optimality of $u$ implies that
\[
u(t,x)\leq u(t+h,y(t+h))+\int^{t+h}_t \ell(y(s),\alpha(s))ds,
\]
where $(y,\alpha)$ satisfies \eqref{DSy}. The definition of $\eta$ implies that
\[
\eta(t+h)\leq u(t,x)-\int^{t+h}_t \ell(y(s),\alpha(s))ds,\ \text{for}\ h\in[0,\tau-t].
\]
Thus, we have
\[
\eta(t+h)\leq u(t+h,y(t+h)).
\]
For any $\phi\in C((0,T)\times\R^d)$ satisfying $\phi\in C^1((0,T)\times\overline{\M_k})$ such that $u-\phi$ attains a local maximum at $(t,x)$, we have
\[
u(t+h,y(t+h))-\phi(t+h,y(t+h))\leq u(t,x)-\phi(t,x).
\]
Then we obtain
\[
\eta(t+h)\leq \phi(t+h,y(t+h))-\phi(t,x)+\eta(t).
\]
Since $y(s)\in \overline{\M_k}$ for $s\in[t,t+h]$, we then deduce that
\[
\dot \eta(t)\leq \partial_t\phi(t,x)+D_{\overline{\M_k}}\phi(t,x)\cdot \dot y(t),
\]
i.e.
\[
-\partial_t\phi(t,x)-D_{\overline{\M_k}}\phi(t,x)\cdot f(x,a)-\ell(x,a)\leq 0.
\]
Now we prove that (iii)$\Rightarrow$(ii).
Since the proof is quite long, we divide it into four steps. In \textbf{Step 1} we treat the trajectories staying in one subdomain  (see Proposition \ref{PROsubmk}). In \textbf{Step 2}  we deal with trajectories  exhibiting a type of "Zeno" effect, i.e crossing the interfaces infinitely during finite time  (see Proposition \ref{PROkey}). 
 In \textbf{Step 3} we deal with the general case  (Proposition \ref{PROPsub2}). Finally in \textbf{Step 4} we conclude the proof of (iii) $\Rightarrow$ (ii).

\begin{step}{{\bf 1}-Trajectories in one subdomain.}

\upshape
\begin{prpstn}\label{PROsubmk}
Assume {\bf(H1)}, {\bf (HF)}, {\bf (HL)}, {\bf(HG)}. Let $u$ be an usc subsolution to \eqref{HJBOmi}-\eqref{HJBHGamma}, $k\in\{0,\ldots,m+l\}$ and $(y(\cdot),\eta(\cdot))$ satisfying \eqref{DIyeta} on some $[a,b]\subset [0,T]$ with $y(s)\in\M_k$ for $s\in[a,b]$.
 Then it holds that
 \[
 u(a,y(a))-\eta(a)\leq u(b,y(b))- \eta(b).
 \]
\end{prpstn}
\begin{proof}
 Using the fact that $y(s)\in\M_k$ for $s\in[a,b]$, then we deduce that
 \[
 \left(\dot y(s),\dot \eta(s)\right)\in G_{\M_k}(y(s)),\ \forall\,s\in (a,b),
 \]
 where $G_{\M_k}$ is Lipschitz continuous. Let $\phi \in C^0((0,T)\times \R^d) \cap C^1((0,T)\times \M_k)$ and $(t,x)$ a local maximum point of $u-\phi$ on $(0,T) \times \M_k$. Since  $u$ is subsolution to \eqref{HJBOmi}, we have on $(0,T)\times\M_k$ 
 \[
 -\partial_t \phi(t,x)+\sup_{a\in A^E_{\M_k}(x)}\{-f(x,a)\cdot D\phi(t,x)-\ell(x,a)\}\leq 0,
 \]
Then, by the definition of $\eta$, we have
 \[
 -\partial_t \phi(t,x)+\sup_{(p,q)\in G_{\M_k}(x)}\{-p\cdot D\phi(t,x)+q\}\leq 0.
 \]
 We set
 \[
 \xi:=u(a,y(a))-\eta(a).
 \]
 By applying \cite[Theorem 4.3.8]{CLSW} for the multifunction $\{1\}\times G_{\M_k}(\cdot)$ and $\Hp(u)\cap(\R\times\M_k\times\R)$, since $(a,y(a),\eta(a)+\xi)\in\Hp(u)\cap(\R\times\M_k\times\R)$ we obtain
 \[
 (s,y(s),\eta(s)+\xi)\in\Hp(u)\cap(\R\times\M_k\times\R)\ \forall\,s\in [a,b].
 \]
 By taking $s=b$ we finally get
 \[
 u(b,y(b))\geq \eta(b)+\xi,
 \]
 which ends the proof.
\end{proof}
\end{step}
\begin{step}{{\bf 2}-"Zeno" type trajectories.}
\upshape
\begin{prpstn}\label{PROkey}
Assume {\bf(H1)}, {\bf (HF)}, {\bf (HL)}, {\bf(HG)}, {\bf (H2)}. Let $u$ be an usc subsolution to \eqref{HJBOmi}-\eqref{HJBHGamma} such  that $u$ is continuous on $\Gamma$ and the restriction of $u$ on $[0,T]\times \Gamma$ is locally Lipschitz continuous. Take $\Gamma_k$ for some $\ k\in\{0,\ldots,l\}$ and $D$ a union of subdomains with $\Gamma_k\subset\overline \D$. 
 Assume that $\D$ enjoys the following property:
 for any $(y(\cdot),\eta(\cdot))$ satisfying \eqref{DIyeta} on some $[a,b]\subset [0,T]$ with $y(s)\in \D$ for $s\in[a,b]$, it holds that
 \begin{equation}\label{inD}
 u(a,y(a))-\eta(a)\leq u(b,y(b))- \eta(b).
 \end{equation}
 Then for any $(y(\cdot),\eta(\cdot))$ satisfying \eqref{DIyeta} on some $[a,b]\subset [0,T]$ with $y(s)\in\D\cup\Gamma_k$ for $s\in[a,b]$, it still holds that 
 \[
 u(a,y(a))-\eta(a)\leq u(b,y(b))- \eta(b).
 \]
\end{prpstn}
\begin{proof}
 Let $(y(\cdot),\eta(\cdot))$ satisfy \eqref{DIyeta} on some $[a,b]\subset [0,T]$ with $y(s)\in\D\cup\Gamma_k$ for $s\in[a,b]$. 
  Without loss of generality, suppose that $y(a)\in\Gamma_k$ and $y(b)\in\Gamma_k$. Otherwise, suppose for example $y(a) \notin \Gamma_k$, then $y(a) \in \D$. We consider the first arrival time $\tau^1$ of $y$ for $\Gamma_k$ and we take $\e>0$ small enough such that
  $$
  [a,\tau^1-\e]\in [a, \tau^1).
  $$
  By \eqref{inD} we have
  $$
  u(a, y(a))-\eta(a)\leq u(\tau^1-\e), y(\tau^1-\e))-\eta(\tau^1-\e)
  $$
  and we conclude sending $\e \to 0$, by the continuity of $y(\cdot), \eta(\cdot) $ and $u(\cdot, \cdot)$.  Analogously, we treat the case $y(b) \in \D$ by considering the last exit time of $y$ for $\Gamma_k$.\\
  We select a compact set $K\subset \R^d$
  containing in its interior the reachable set
  $$
  R_{G_{\Gamma_k}}
  (y([a, b]) \cap \Gamma_k, b)=\bigcup_{t\in[a,b]} \{x \in \R^d\, |\, \exists \quad \mbox{traj. } w \mbox{ of } G_{\Gamma_k} \mbox{ with } w(a) \in y([a, b]) \cap \Gamma_k, w(t) = x\}.
  $$
 We denote by $\Gamma^\natural_k$ an open neighbourhood of $\Gamma_k$ such that $y([a,b])\subset \Gamma^\natural_k$ and we introduce the following notations that will appear in the
  forthcoming estimates.
  \begin{itemize}
  	\item $L_u$ is the Lipschitz constants of $u$ in $(K \cap \Gamma_k) \times [0,T]$;
  	\item $M$ estimates from above the diameter of $G_{\Gamma_k}(x)$ for $x \in R_{G_{\Gamma_k}} (y[a, b] \cap \Gamma_k, b);$
  	\item $L_G$ is a Lipschitz constant for $G_{\Gamma_k}$ (suitably extended outside the interfaces, see Corollary $A.2$ of \cite{RSZ}) in $K \cap \Gamma^\natural_k$.
  \end{itemize}
  By \textbf{(H1)}, $\Gamma_k\cap\D=\emptyset$.  

 Let $J:=\{s\in [a,b]\,:\,y(s)\not\in \Gamma_k\}$, then $J$ is an open set and can be written as the unions of disjoint intervals:
 \[
 J=\bigcup^{\infty}_{n=1}(a_n,b_n).
 \]
 For a fixed $p\in\N$, we set
 \[
 J_p:=\bigcup^p_{n=1}(a_n,b_n)
 \]
 as the union of the first $p$ intervals. After reindexing, we assume without loss of generality that
 \[
 a_1<b_1\leq a_2<b_2\leq \cdots\leq a_p<b_p.
 \]
 We set $b_0:=a$ and $a_{p+1}:=b$, and we choose $p$ sufficiently large such that
 \[
 \meas(J\backslash J_p)<\frac{r}{2Me^{LT}},
 \]
 where $r>0$ is given by
 \[
 r:=\inf\{\|y(s)-z\|\,:\, s\in[a,b],\ z\in\overline{\Gamma_k}\backslash \Gamma_k\}.
 \]
 
 At first, we focus on the part of $y(\cdot)$ restricted on $[a_n,b_n]$ for $n=1,\ldots,p$. Note that for $s\in(a_n,b_n)$, $y(s)\in\D$. Let $\e>0$ small enough such that
 \[
 [a_n+\e,b_n-\e]\subset (a_n,b_n),
 \]
 then by the assumption, it follows that
 \[
 u(a_n+\e,y(a_n+\e))-\eta(a_n+\e)\leq u(b_n-\e,y(b_n-\e))- \eta(b_n-\e).
 \]
 By the continuity of $y(\cdot)$, $\eta(\cdot)$ and $u(\cdot,\cdot)$, we obtain by setting $\e \to 0$
 \[
 u(a_n,y(a_n))-\eta(a_n)\leq u(b_n,y(b_n))- \eta(b_n).
 \]
 The next step is to deal with the part of $y(\cdot)$ restricted on $[b_n,a_{n+1}]$ for $n=0,\ldots,p$. We set $\e_n:=\meas([b_n,a_{n+1}]\cap J)$, then $\sum^p_{n=0}\e_n=\meas(J\backslash J_p)$. 
 
 For any $s\in [b_n,a_{n+1}]\backslash J$, $y(s)\in \Gamma_k$. It follows that
 \[
 (\dot y(s),\dot \eta(s))\in G_{\Gamma_k}(y(s))\ \text{a.e.}\ s\in [b_n,a_{n+1}]\backslash J.
 \]
 Now we calculate how far $(y(\cdot),\eta(\cdot))$ is from any trajectory lying in $\Gamma_k$ driven by the dynamics $G_{\Gamma_k}$ by
 \[
 \xi_n:=\int^{a_{n+1}}_{b_n}\dist\left((\dot y(s),\dot \eta(s)),G_{\Gamma_k}(y(s))\right)ds\leq 2M\e_n.
 \]
 By Proposition \ref{PROPGM}, $G_{\Gamma_k}$ is locally Lipschitz. Then we can apply Filippov's Theorem (see \cite{C}, Theorem $3.1.6$ and also \cite{CLSW}, Proposition $3.2$ and we get that there exists $(z_n,\zeta_n)$ satisfying
 \[
 (\dot z_n(s),\dot \zeta_n(s))\in G_{\Gamma_k}(z_n(s)),\ \text{a.e.}\ s\in [b_n,a_{n+1}]
 \]
 with $(z_n(b_n),\zeta_n(b_n))=(y(b_n),\eta(b_n))$, and 
 \[
 \|(z_n(a_{n+1}),\zeta_n(a_{n+1}))-(y(a_{n+1}),\eta(a_{n+1}))\|\leq e^{L_G(a_{n+1}-b_n)}\xi_n\leq 2Me^{L_G(a_{n+1}-b_n)}\e_n.
 \]
 From the above properties of $z_n$ and the choice of $p$, we observe that $z_n(s)\in\Gamma_k$ for $s\in[b_n,a_{n+1}]$. Thus, from Proposition \ref{PROsubmk} one obtains
 \[
 u(b_n,z_n(b_n))-\zeta_n(b_n)\leq u(a_{n+1},z_n(a_{n+1}))-\zeta_n(a_{n+1}).
 \]
 This implies for $p$ big enough
 \[
 u(b_n,y(b_n))-\eta(b_n)\leq u(a_{n+1},y(a_{n+1}))-\eta(a_{n+1})+2M(L_u+1)e^{L_G(a_{n+1}-b_n)}\e_n.
 \]
   Then for $n=0,\ldots,p$, we deduce that
 \[
 u(a_n,y(a_n))-\eta(a_n)\leq u(a_{n+1},y(a_{n+1}))-\eta(a_{n+1})+2M(L_u+1)e^{L_G(a_{n+1}-b_n)}\e_n.
 \]

 Finally,
 \begin{eqnarray*}
  u(a,y(a))-\eta(a) &=& u(b_0,y(b_0))-\eta(b_0) \\
  &\leq& u(a_1,y(a_1))-\eta(a_1)+2M(L_u+1)e^{L_G(a_1-a)}\e_0 \\
  &\leq& u(a_2,y(a_2))-\eta(a_2)+2M(L_u+1)e^{L_G(a_2-a)}(\e_0+\e_1) \\
  &\cdots& \\
  &\leq& u(a_{p+1},y(a_{p+1}))-\eta(a_{p+1})+2M(L_u+1)e^{L_G(a_{p+1}-a)}\sum^p_{n=0}\e_n \\
  &=&  u(b,y(b))-\eta(b)+2M(L_u+1)e^{L_G(b-a)}\meas(J\backslash J_p).
 \end{eqnarray*}
 By taking $p\ra +\infty$, one has $\meas(J\backslash J_p)\ra 0$ and the desired result is obtained.
\end{proof}
\end{step}
\begin{step}{{\bf 3}-General case.}
\upshape
\begin{prpstn}\label{PROPsub2}
Assume {\bf(H1)}, {\bf (HF)}, {\bf (HL)}, {\bf(HG)}, {\bf (H2)}. Let $u$ be an usc subsolution to \eqref{HJBOmi}-\eqref{HJBHGamma}. If $u$ is continuous on $[0,T]\times\Gamma$ and the restriction of $u$ on $[0,T]\times \Gamma$ is locally Lipschitz continuous, then
 for any $(y(\cdot),\eta(\cdot))$ satisfying \eqref{DIyeta} on some $[a,b]\subset [0,T]$, it holds that
 \begin{equation}\label{eqclaim}
 u(a,y(a))-\eta(a)\leq u(b,y(b))- \eta(b).
 \end{equation}
\end{prpstn}
\begin{proof}
 Let $\D$ be a union of some subdomains and $d_ \D\in\{0,\ldots,d\}$ be the minimal dimension of the subdomains which are subsets of $\D$. The proof of \eqref{eqclaim} is based on the following induction argument with regard to $d_ \D$:
 \begin{description}
  \item[\textbf{Claim}]: for any $\tilde d=0,\ldots,d$, any $\D$ with $d_ \D\geq \tilde d$ and any $(y,\eta)$ driven by $G$ with $y$ lying within $\D$, \eqref{eqclaim} holds.
 \end{description}
 
 \medskip
Let us first check the case when $\tilde d=d$. In this case, $d_ \D=d$, then $\D$ is a union of $d$-manifolds, which are disjoint by {\bf (H1)}. For any trajectory $(y,\eta)$ driven by $G$ with $y$ lying within $\D$, $y$ lies entirely within one of the $d$-manifolds. Hence,  {\bf Claim} follows by Proposition \ref{PROsubmk}.
 
 Now we assume that {\bf Claim} is true for some $\tilde d\in\{1,\ldots,d\}$ and we prove that {\bf Claim} still holds true for $\tilde d-1$. In this case, $d_ \D=\tilde d-1$. Then the following three cases can occur.
 
{\bf Case 1}: if $\D$ contains only one subdomain, i.e. $\D=\M_k$ for some $k\in\{0,\ldots,l+m\}$, by Proposition \ref{PROsubmk} it follows that {\bf Claim} holds.

{\bf Case 2}: If $\D$ contains more than one subdomain and $\D$ is connected, let $\M_1',\ldots,\M_p'$ be all the subdomains of $\D$ with the dimension $d_ \D$. Then $\K:=\D\backslash(\cup^p_{k=1}\M_k')$ is a union of subdomains with dimension greater than $\tilde d$. As an induction hypothesis, \eqref{eqclaim} holds true for any $(y,\eta)$ driven by $G$ with $y$ lying within $\K$. 
  
  Now note that, for all $i \in \{1,\cdots, p\}$ there exists some $k$ such that $\M_i'= \Gamma_k$ and $\M_i'\subset\overline \K$. Then Proposition \ref{PROkey} implies that \eqref{eqclaim} holds true for any $(y,\eta)$ driven by $G$ with $y$ lying within $\K\cup\M_1'$. We continue applying Proposition \ref{PROkey} for $\K\cup\M_1'$ and $\M_2'$ until $\K\cup\M_1'\cup\cdots\cup\M_{p-1}')$ and $\M_p'$, finally it is obtained that \eqref{eqclaim} holds true for any $(y,\eta)$ driven by $G$ with $y$ lying within $\D$ ($=\K\cup\M_1'\cup\cdots\cup\M_{p}'$).

 {\bf Case 3}: If $\D$ is not connected, for any $(y,\eta)$ driven by $G$ with $y$ lying within $\D$, $y$ lies within one connected component of $\D$ since $y$ is continuous. Then the proof follows the same argument as in the above case. And the induction step is complete.\\
 Finally, we conclude the proof by taking $\D=\R^d$ with $d_ \D$ being the dimension of $\Gamma_0$.

\end{proof}
\end{step}
\begin{step}{{\bf 4}}
\upshape
Finally we conclude the proof of (iii) $\Rightarrow$ (i).
 For any $(t,x)\in[0,T]\times\R^d$, and any $(y(\cdot),\alpha(\cdot))$ satisfying \eqref{DSy}, we set
 \[
 \eta(s):=u(t,x)-\int^{s}_t \ell(y(s'),\alpha(s'))ds',\ \text{for}\ s\in[t,T].
 \]
 Since $(y(\cdot),\eta(\cdot))$ satisfies \eqref{DIyeta}, Proposition \ref{PROPsub2} implies that
 \[
 u(t,x)-\eta(t)\leq u(t+h,y(t+h))-\eta(t+h),\ \forall\,h\in [0,T-t],
 \]
 i.e.
 \[
 0\leq u(t+h,y(t+h))-u(t,x)+\int^{t+h}_t \ell(y(s),\alpha(s))ds,\ \forall\,h\in [0,T-t],
 \]
 which ends the proof.
 \end{step}
\end{proof}
\section{Proof of the main results}\label{mres}
 In this section we prove our main results, that is Theorem \ref{THMHEHGamma} and Theorem \ref{THMBilateral}. First we prove Theorem \ref{THMHEHGamma}.

\begin{proof}[Proof of Theorem \ref{THMHEHGamma}]
\begin{step}{{\bf 1}}
\upshape
First we prove the following  comparison principle. Let $u_1,u_2:[0,T]\times\R^d\ra\R$ be respectively a supersolution and subsolution to \eqref{HJBOmi}-\eqref{HJBHE} (\eqref{HJBOmi}-\eqref{HJBHGamma}) with $u_1(T,\cdot)  \leq u_2(T,\cdot)$. Assume, in addition, that $u_1$ is continuous at any point of $\Gamma$. Then 
 \[
 u_1(t,x)\leq u_2(t,x),\ \forall\,t\in(0,T),\ x\in\R^d.
 \]
Indeed, by Theorem \ref{CHARAsuper}, $u_2$ satisfies the super-optimality, i.e. there exists $\bar y,\bar \alpha$ such that
 \[
 u_2(t,x)\geq u_2(T,\bar y(T))+\int^T_t\ell(\bar y(s),\bar \alpha(s))ds.
 \]
 By Theorem \ref{CHARAsub}, $u_1$ satisfies the sub-optimality. Then we have
 \[
 u_1(t,x)\leq u_1(T,\bar y(T))+\int^T_t\ell(\bar y(s),\bar \alpha(s))ds.
 \]
 Then we deduce that
 \[
 u_1(t,x)-u_2(t,x)\leq u_1(T,\bar y(T))-u_2(T,\bar y(T))\leq 0.
\]
\end{step}
\begin{step}{{\bf 2}}
\upshape
Now we prove that the value function $v$ is the unique continuous viscosity solution to \eqref{HJBOmi}-\eqref{HJBHE} (\eqref{HJBOmi}-\eqref{HJBHGamma}) with $\lambda_l$-superlinear growth. The continuity and  the $\lambda_l$-superlinear growth of $v$ are given respectively in  Proposition \ref{contv} and Proposition \ref{lambdagr}. In addition, the restriction of $v$ on $[0,T]\times \Gamma$ is locally Lipschitz continuous (see Appendix $A$, Proposition \ref{contv}).  Then, by Proposition \ref{PROdpp}, Theorem \ref{CHARAsuper} and Theorem \ref{CHARAsub}, $v$ is a viscosity solution to \eqref{HJBOmi}-\eqref{HJBHE} (\eqref{HJBOmi}-\eqref{HJBHGamma}) with  $v(T,x)=\vp(x)$ for all $x \in \R^d$. The uniqueness of $v$ follows by \textbf{Step 1}.
\end{step}
\end{proof}

Now we prove Theorem \ref{THMBilateral}.
\begin{proof}[Proof of Theorem \ref{THMBilateral}]
We proceed as in Theorem \ref{THMHEHGamma} and first we prove the following comparison principle. Let $u_1,u_2:[0,T]\times\R^d\ra\R$ are respectively a supersolution and bilateral subsolution to \eqref{HJBOmi}-\eqref{HJBHGamma} with $u_1(T,\cdot)\leq u_2(T,\cdot)$. Then 
 \begin{equation}\label{comp2}
 u_1(t,x)\leq u_2(t,x),\ \forall\,t\in(0,T),\ x\in\R^d.
 \end{equation}
 We omit the proof of \eqref{comp2} since it follows as in Theorem \ref{THMHEHGamma} by using Theorem \ref{CHARAsubback} instead of Theorem \ref{CHARAsub}.
 Next, thanks to Proposition \ref{lambdagr}, the value function $v$ is lsc with $\lambda$-superlinear growth. Also, it is a bilateral viscosity solution to \eqref{HJBOmi}-\eqref{HJBHGamma} due to Proposition \ref{PROdpp}, Theorem \ref{CHARAsuper} and Theorem \ref{CHARAsub}. Moreover $v(T,x)=\vp(x)$ for all $x \in \R^d$. The uniqueness of a lsc bilateral solution to \eqref{HJBOmi}-\eqref{HJBHGamma} follows by the comparison principle.
\end{proof}

\section{Stability result}\label{sres}
Let $(f^n)_{n \in \N}$ and $(\ell^n)_{n \in \N}$ be a sequence of functions defined on $\R^d \times \As$ such that
$$
f^n\ra f,\,\, \ell^n\ra \ell\ \text{locally uniformly in } \R^d\times \As.
$$
For convenience of notation we denote for each $n \in \N$
$$
F^n(x)=\{f^n(x,a)\, |\, a \in \As\},\, \quad L^n(x)=\{\ell^n(x,a)\, |\, a \in \As\}
$$
and we suppose that $F^n$, $L^n$ satisfies {\bf (HF)-(HL)} with constants uniform in $n$. 

As in Definition \ref{DefFE} and Definition \ref{DEFG}, we redefine the essential dynamics $F^{E,n}$, the augmented dynamics $G^n,G^n_{\M_k}$ and the essential control sets $A^{E,n}$ and $A^{E,n}_{\M_k}$ for each $\M_k\in\Ms$. 

We set
\[
H^{E,n}(x,p):=\sup_{a\in A^{E,n}(x)}\{-f^n(x,a)\cdot p-\ell^n(x,a)\},
\]
and consider the following equation:
\begin{equation}\label{HJBn}
-\partial_t u_n+H^{E,n}(x,Du_n)=0.
\end{equation}

We have the following stability result for the supersolutions.
\begin{thrm}\label{Stabilitysuper}
	Assume {\bf (H1)}, {\bf(HF)}, {\bf(HL)}, {\bf (HG)}, {\bf (H2)}. 
	If $u_n$ is a lsc supersolution to
	\[
	-\partial_t u_n+H^{E,n}(x,Du_n)=0,
	\]
	and $u_n$ converges to a lsc function $u$ locally uniformly in $[0,T]\times\R^d$, then $u$ is a supersolution to \eqref{HJBOmi}-\eqref{HJBHE}.
\end{thrm}
\begin{proof}
	By Theorem \ref{CHARAsuper}, it suffices to prove that $u$ is a supersolution to \eqref{HJBOmi}-\eqref{HJBHGamma}.

	For any $t\in(0,T)$, $x\in\R^d$, $\phi\in C^1((0,T)\times\R^d)$ such that $u-\phi$ attains a local strict minimum at $(t,x)$, then there exists $t_n\in (0,T)$, $x_n\in\R^d$ such that $u_n-\phi$ attains a local minimum at $(t_n,x_n)$ with $t_n\ra t$, $x_n\ra x$. Thus,
	\[
	-\partial_t\phi(t_n,x_n)+\sup_{(p,q)\in G_n(x_n)}\{-p\cdot D\phi(t_n,x_n)+q\}\geq 0.
	\]
	For any $\e>0$, since $G^n\ra G$ locally uniformly and $G$ is usc, for $n$ sufficiently large we have
	\[
	G^n(x_n)\subset G(x_n)+\e B(0,1)\subset G(x)+2\e B(0,1).
	\]
	Then there exists $C>0$ such that
	\[
	-\partial_t\phi(t_n,x_n)+\sup_{(p,q)\in G(x)}\{-p\cdot D\phi(t_n,x_n)+q\}+C\e\geq 0.
	\]
	By taking $n\ra\infty$ then $\e\ra 0$, we obtain
	\[
	-\partial_t\phi(t,x)+\sup_{(p,q)\in G(x)}\{-p\cdot D\phi(t,x)+q\}\geq 0.
	\]
	The definition of $G$ then implies that
	\[
	-\partial_t\phi(t,x)+\sup_{a\in \As}\{-f(x,a)\cdot D\phi(t,x)-\ell(x,a)\}\geq 0.
	\]
	Therefore $u$ is a supersolution to \eqref{HJBOmi}-\eqref{HJBHGamma}.
\end{proof}
The stability result for the subsolutions is the following.
\begin{thrm}\label{Stabilitysub}
	Assume {\bf(HF)}, {\bf(HL)}, {\bf (HG)}, {\bf (H2)}. 
	If $u_n$ is a usc subsolution to
	\[
	-\partial_t u_n+H^{E,n}(x,Du_n)=0,
	\]
	$u_n|_{[0,T]\times\overline{\M_k}}$ is locally Lipschitz continuous and $u_n$ converges to an usc function $u$ locally uniformly in $[0,T]\times\R^d$, then $u$ is a subsolution to \eqref{HJBOmi}-\eqref{HJBHE}.
\end{thrm}
\begin{proof}
	Note that in $(0,T)\times\Om_i$, $i=1,\ldots,m$, the proof follows from the standard arguments for stability results on viscosity solutions since $H^{E,n}$ and $H^E$ are Lipschitz continuous in $\Om_i\times\R^d$.
	
	For any $t\in(0,T)$, $x\in\Gamma$, $\phi\in C((0,T)\times\R^d)$, $\phi\in C^1((0,T)\times\overline{\M_k})$ with $x\in\overline{\M_k}$ such that $u-\phi$ attains a local strict maximum at $(t,x)$, then $(t,x)\mapsto u(t,x)-\phi(t,x)-Cd_{\overline{\M_k}}(x)$ also attains a local strict maximum at $(t,x)$ for any constant $C>0$. Since $u_n\ra u$, there exists $t_n\in (0,T)$, $x_n\in\R^d$ such that $u_n-\phi-d_{\overline{\M_k}}$ attains a local maximum at $(t_n,x_n)$ with $t_n\ra t$, $x_n\ra x$. 
	
	\smallskip
	
	We claim that with a big enough $C$, $$x_n\in\overline{\M_k}.$$
	
	If $x_n\in\Om_i$ for some $i\in\{1,\ldots,m\}$, note that $d_{\overline{\M_k}}(\cdot)$ is differentiable in $\Om_i$. Since $u_n$ is a subsolution to \eqref{HJBn}, we have by choosing $C=n$
	\[
	-\partial_t\phi(t_n,x_n)+\sup_{a\in \As}\left\{-f^n_i(x_n,a)\cdot \left(D\phi(t_n,x_n)+n\frac{x_n-\Proj_{\overline{\M_k}}(x_n)}{|x_n-\Proj_{\overline{\M_k}}(x_n)|}\right)-\ell_i^n(x_n,a)\right\}\leq 0.
	\]
	Because of {\bf (H2)}, the above inequality does not hold true when $n$ is big enough. Then we conclude that $x_n \in \Gamma$.

	Now for any $z\in\Gamma$ close to $x$, using the fact that $u_n|_{[0,T]\times\M_k}$ is locally Lipschitz continuous,
	\begin{eqnarray*}
		&&u_n(t,z)-\phi(t,z)-C d_{\overline{\M_k}}(z)\\
		&\leq& u_n(t,P_{\overline{\M_k}}(z))-\phi(t,P_{\overline{\M_k}}(z))+(L_{u_n}+L_ \phi)d_{\overline{\M_k}}(z)-C d_{\overline{\M_k}}(z),
	\end{eqnarray*}
	where $L_{u_n},L_ \phi$ are respectively the local Lipschitz constants of $u_n$ and $\phi$. Since $L_{u_n}$ are uniform in $n$, we
	can take $C>L_{u_n}+L_ \phi$, then
	\[
	u_n(t,z)-\phi(t,z)-C d_{\overline{\M_k}}(z)\leq u_n(t,P_{\overline{\M_k}}(z))-\phi(t,P_{\overline{\M_k}}(z)),
	\]
	which implies that $x_n\in \overline{\M_k}$, and the claim is proved. 
	
	Since $u_n$ is a subsolution to \eqref{HJBOmi}-\eqref{HJBHE}, then
	\[
	-\partial_t \phi(t_n,x_n)+\sup_{a\in A^{E,n}_{\M_k}(x_n)}\{-f^n(x_n,a)\cdot D_{\overline{\M_k}}\phi(t_n,x_n)-\ell^n(x_n,a)\}\leq 0,
	\]
	which by the definition of $A^{E,n}_{\M_k}(x_n)$ and the augmented dynamics is equivalent to
	\[
	-\partial_t \phi(t_n,x_n)+\sup_{(p,q)\in G^{n}_{\M_k}(x_n)\cap\left(\T_{\overline{\M_k}}(x_n)\times\R\right)}\{-p\cdot D_{\overline{\M_k}}\phi(t_n,x_n)+q\}\leq 0,
	\]
	Note that $\phi\in C^1((0,T)\times\overline{\M_k})$, one has
	\[
	\partial_t \phi(t_n,x_n)\ra\partial_t \phi(t,x)\ \text{and}\ D_{\overline{\M_k}}\phi(t_n,x_n)\ra D_{\overline{\M_k}}\phi(t,x)\ \text{when}\ n\ra\infty.
	\]
	By the Lipschitz continuity of $F^n,L^n$ uniformly in $n$, we deduce that there exists some constant $L>0$ such that
	\[
	G^{n}_{\M_k}(x)\subset G^{n}_{\M_k}(x_n)+LB(0,|x_n-x|).
	\]
	Therefore for any $\e>0$, there exists $N_1\in\N$ such that for all $n\geq N_1$
	\[
	-\partial_t \phi(t,x)+\sup_{(p,q)\in G^{n}_{\M_k}(x)\cap\left(\T_{\overline{\M_k}}(x_n)\times\R\right)}\{-p\cdot D_{\overline{\M_k}}\phi(t,x)+q\}\leq \e.
	\]
	By the local uniform convergence of $f^n$ and $\ell^n$ in $\R^d\times\As$, we obtain $G^n_{\M_k}(x)\ra G_{\M_k}(x)$ with respect to the Hausdorff metric. Thus, for any $\e>0$, there exists $N_2\in\N$ such that for all $n\geq N_2$
	\[
	-\partial_t \phi(t,x)+\sup_{(p,q)\in G_{\M_k}(x)\cap\left(\T_{\overline{\M_k}}(x_n)\times\R\right)}\{-p\cdot D_{\overline{\M_k}}\phi(t,x)+q\}\leq \e.
	\]
	Besides, note that $x_n,x\in\overline{\M_k}$ and $x_n\ra x$, it holds that for $n$ sufficiently large
	\[
	\T_{\overline{\M_k}}(x)\subset \T_{\overline{\M_k}}(x_n).
	\]
	Consequently, for any $\e> 0$, there exists $N_3\in\N$ such that for all $n\geq N_3$
	\[
	-\partial_t \phi(t,x)+\sup_{(p,q)\in G_{\M_k}(x)\cap\left(\T_{\overline{\M_k}}(x)\times\R\right)}\{-p\cdot D_{\overline{\M_k}}\phi(t,x)+q\}\leq \e.
	\]
	The above inequality holds for arbitrary $\e> 0$, therefore
	\[
	-\partial_t \phi(t,x)+\sup_{(p,q)\in G_{\M_k}(x)\cap\left(\T_{\overline{\M_k}}(x)\times\R\right)}\{-p\cdot D_{\overline{\M_k}}\phi(t,x)+q\}\leq 0,
	\]
	which is equivalent to
	\[
	-\partial_t \phi(t_n,x_n)+\sup_{a\in A^{E}_{\M_k}(x)}\{-f(x,a)\cdot D_{\overline{\M_k}}\phi(t,x)-\ell(x,a)\}\leq 0.
	\]
	We then conclude that $u$ is a subsolution to \eqref{HJBOmi}-\eqref{HJBHE}.
\end{proof}
Finally, we provide the stability result with respect to the final cost $\vp$.
\begin{thrm}
	Assume {\bf (H$\vp$1)} {\bf(HF)}, {\bf(HL)}, {\bf (HG)}, {\bf (H2)}. Let $\vp_n:\R^d\times\R$ be a sequence of Lipschitz continuous functions, such that $\vp_n\ra\vp$ locally uniformly in $\R^d$.
	Let $u_n$ be the solution to
	\begin{equation}\label{HJBn}
	\left\{
	\begin{array}{ll}
	-\partial_t u_n(t,x)+H^{E,n}(x,Du_n(t,x))=0 & \text{in}\ (0,T)\times\R^d,\\
	u_n(T,x)=\vp_n(x) & \text{in}\ \R^d,
	\end{array}
	\right.
	\end{equation}
	 such that the restriction of $u_n$ to $[0,T]\times \Gamma$ is locally Lipschitz continuous. If $u_n$ converges to a continuous function $u$ locally uniformly in $[0,T]\times\R^d$, then $u$ is the solution to \eqref{HJBOmi}-\eqref{HJBHE}.
\end{thrm}
\begin{proof}
 By Theorem \ref{Stabilitysuper} and Theorem \ref{Stabilitysub}, $u$ is a supersolution and a subsolution of \eqref{HJBOmi}-\eqref{HJBHE}. Besides,  for any $x\in\R^d$,
 \[
 u(T,x)=\lim_{n\ra\infty} u_n(T,x)=\lim_{n\ra\infty} \vp_n(T,x)=\vp(T,x),
 \]
 i.e. $u$ satisfies the final condition. Thus, $u$ is the solution to \eqref{HJBOmi}-\eqref{HJBHE}.
\end{proof}

\appendix
\renewcommand\thesection{\Alph{section}}
\renewcommand\theequation{\Alph{section}.\arabic{equation}}
\section{Appendix A}
Let us start by the proof of Proposition \ref{PROPGM}.  
\begin{proof1}
\upshape
For $k \in \{1,\cdots, m+l\}$,  consider the subdomain $\M_k$, which in the following proof we denote by $\M$ for simplicity. We consider the following three cases according to the dimension of $\M$.

\smallskip

{\bf Case 1}: $\M_k=\Om_i$, $i\in\{1,\ldots,m\}$.

The claim simply follows by noting that in this case $\T_{\M_k}(x)=\R^d$,
and then $G_\M$ is locally Lipschitz continuous since $f(x,a)$ and $\ell(x,a)$ are locally Lipschitz continuous for each $a \in \As$.

\smallskip

{\bf Case 2}:   $\M\in \Gamma \setminus (\Gamma^0_k)_{k=1,\ldots q\times q}$. 

Note that the proof follows the main ideas of  \cite{RSZ}, Theorem $A.1$. Nevertheless, we give the proof for completeness and for a better understanding of {\bf Case $3$}.

 We want to show the existence of $\bar{L}>0$ such that for any $K$  compact of $\M$, $x,z \in K$ and $(f(x,a),q_1) \in G_\M(x), q_1 \leq -\ell(x,a)$ there exist a control $c$ and   $(f(z,c),q_2) \in G_\M(z), q_2 \leq -\ell(z,c)$ such that
\begin{equation}
\label{claimlipmult}
|f(x,a)-f(z,c)|+|q_1-q_2|\leq \bar{L}|x-z|,
\end{equation}
or, equivalently,
\[\label{lipG}
G_\M(x)\subset G_\M(z)+\bar{L}B(0,|x-z|).
\]
It is not restrictive to prove the above inequality for $|x-z|$ small, therefore since there are just a finite
number of connected components of $K$ intersecting $\Gamma$ and such components are at a positive distance apart, we can assume, without lose of generality, that $\Gamma$ is  connected.
We denote by $n$  the exterior normal vector to $\M$ as defined in subsection \eqref{secass}. Then, by the regularity of $\Gamma$ and since $K$ is connected,  $n$ is Lipschitz continuous.
To simplify the notations, we denote by $L$ the Lipschitz constant of  $f, \ell$ and $n$ in $K$. Moreover, we denote by $M$ a constant estimating form above $|n|$ in $K$ and
and $|f|, |\ell|$ in $K\times \As$.

Also, since $f(x,a) \in \T_\M(x)$ on $\M$ we have
\begin{equation}\label{fxtan}
f(x,a)\cdot n(x)=0.
\end{equation}
Since $f, \ell$ are locally Lipschitz we have
\begin{equation}\label{floclip}
|f(x,a)-f(z,a)|\leq L|x-z|,\,\, |\ell(x,a)-\ell(z,a)|\leq L|x-z|.\ 
\end{equation}
Note that, if $f(z,a)\cdot n(z)=0$, then $f(z,a)\in\mathcal{T}_{\M}(z)$ and by \eqref{fxtan} we deduce
\eqref{claimlipmult}
with $\bar{L}=2L$.
Suppose now that
\begin{equation}
\label{beta}
f(z,a)\cdot n(z):=-\beta<0.
\end{equation}
The controllability assumption {\bf (H2)} implies in particular that there exists $b\in \As$ such that
\begin{equation}\label{gamma}
f(z,b)\cdot n(z):=\gamma>0.
\end{equation}
Note  that by \eqref{fxtan} and \eqref{floclip}, we have
\[
\beta=|f(z,a)\cdot  n(z)|\leq (L+ML)|x-z|:=C|x-z|,
\]
and then
\begin{equation}\label{cxzbeta}
|x-z|\geq \frac{\beta}{C}.
\end{equation}
By the convexity assumption {\bf (HG)} for $G(z)$, there exists $c\in \As$, $\tilde q\in \R$ such that
\begin{equation}\label{conveq}
(f(z,c),\tilde q)=\frac{\gamma}{\beta+\gamma} \left(f(z,a),-\ell(z,a)\right)+\frac{\beta}{\beta+\gamma} (f(z,b),-\ell(z,b)),\, \,\, \tilde q\leq -\ell(z,c).
\end{equation}
Then, we show that \eqref{claimlipmult} holds for such $c \in \As$.
Indeed, by \eqref{conveq} and \eqref{cxzbeta}, we get
\begin{equation}\label{fzafzc}
|f(z,a)-f(z,c)|\leq 2M\frac{\beta}{\beta+\gamma}\leq \frac{2M}{\gamma}\beta\leq \frac{2MC}{\gamma}|x-z|,
\end{equation}
and by the first of \eqref{floclip} and \eqref{fzafzc} we conclude
\[
|f(x,a)-f(z,c)|\leq |f(x,a)-f(z,a)|+|f(z,a)-f(z,c)|\leq \left(L+\frac{2MC}{\gamma}\right)|x-z|.
\]
Moreover by \eqref{beta} and \eqref{gamma}, we have
\[
f(z,c)\cdot n(z)=\frac{\gamma}{\beta+\gamma} f(z,a)\cdot n(z)+\frac{\beta}{\beta+\gamma}f(z,b)\cdot n(z)=0,
\]
which implies that $f(z,c)\in \T_{\M}(z)$.

Observe that
\begin{eqnarray*}
	\tilde q &=& -\frac{\gamma}{\beta+\gamma}\ell(z,a)-\frac{\beta}{\beta+\gamma}\ell(z,b) \\
	&\geq& -\frac{\gamma}{\beta+\gamma}(\ell(x,a)+L|x-z|)-\frac{\beta}{\beta+\gamma}\ell(z,b) \\
	&\geq& -\ell(x,a)- L|x-z|+\frac{\beta}{\beta+\gamma}\left(\ell(x,a)-\ell(z,b)\right) \\
	&\geq& q_1-L|x-z|-2M\frac{\beta}{\gamma} \\
	&\geq& q_1-\left(L+\frac{2MC}{\gamma}\right)|x-z|,
\end{eqnarray*}
where we used the definition of $q_x$ and \eqref{cxzbeta}.
If we set $q_2:=q_1-(L+2MC/\gamma)|x-z|$, then we have
\[
q_2\leq \tilde q\leq -\ell(z,c)
\]
 and therefore
\[
(f(z,c),q_2)\in G_\M(z),\ |(f(x,a),q_1)-(f(z,c),q_2)|\leq (L+\frac{2MC}{\gamma})|x-z|.
\]
Then we conclude that
\[
G_\M(x)\subset G_\M(z)+(L+\frac{2MC}{\gamma})B(0,|x-z|).
\]
\smallskip

{\bf Case 3}: $\M\in (\Gamma^0_k)_k$. Denote for simplicity $\M=\Gamma^0$ and let $\Hy_{j_1}, \Hy_{j_2}$ be such that $\M=\Hy_{j_1}\cap \Hy_{j_2}$. 
  The proof essentially follows  by noting that $\T_{\Gamma_0}(\cdot)=\T_{\Hy_{j_1}}(\cdot)\cap\T_{\Hy_{j_2}}(\cdot)$ and by applying the same arguments used in {\bf Case $2$}. We just give a sketch of the main steps.
 
We want to show that, given a compact $K$ and $x,z \in K$, $(f(x,a),q_1) \in G_{\Gamma_0}(x), q_1 \leq -\ell(x,a)$, there exist a control $c \in \As$ and   $(f(z,c),q_2) \in G_{\Gamma_0}(z), q_2 \leq - \ell(z, c)$ such that
\begin{equation}
\label{claimlipmult2}
|f(x,a)-f(z,c)|+|q_1-q_2|\leq \bar{L}|x-z|.
\end{equation}
Note that the condition $f(z,c) \in \T_{\Gamma_0}$ now reads $f(z,c)  \in \T_{\Hy{j_1}}(z)\cap \T_{\Hy{j_2}}(z)$, that is $f(z,c) \cdot n_1(z)=0, f(z,c)\cdot n_2(z)=0,$ where $n_1(z),n_2(z)$ are the normal respectively to $\Hy_{j_1}$ and $\Hy_{j_2}$ as defined in subsection \ref{secass}. By {\bf Case 2} we can suppose that  $f(z,a) \cdot n_1(z)=0$ without loss of generality. Suppose that $f(z,a) \cdot n_2(z)=-\beta<0$ as in \eqref{beta}. Then we proceed analogously as in {\bf Case 1} using the controllability assumption ${\bf (H2)}$ on $\Hy_{j_2}$ and we find $b\in \As$ such that
$$
f(z,b)=n_2(z)\gamma, \, \, \gamma>0.
$$
Then, by the convexity of $G(z)$, we find a control $c \in \As$ such that 
$$
f(z,c)=\frac{\gamma}{\beta+\gamma}f(z,a)+\frac{\beta}{\beta+\gamma} f(z,b)
$$
and then we conclude
$$
f(z,c) \cdot n_2(z)=0, \, \, f(z,c) \cdot n_1(z)=0.
$$
The rest of the proof can be carried out exactly as in {\bf Case 1} and we omit the details.
Finally, not that the case $f(z,a) \cdot n_1(z)>0$ or $f(z,a) \cdot n_1(z)<0$ can be treated analogously.
\end{proof1}
Now we prove Proposition \ref{PROPH3}.
\begin{proofH3}\begin{normalfont}
We start by proving $(i)$.
  For each $j=0,\ldots,l$, if $F(x)\cap \mathcal T_{\Gamma_j}(x)=\emptyset$ for any $x\in\Gamma_j$, then $G_{\Gamma_j}(x)=\emptyset$ by definition. Otherwise if there exists $r_2>0$ such that for any $x\in\Gamma_j$, $B(0,r_2)\subset F(x)$, then $G_{\Gamma_j}$ is locally Lipschitz continuous on $\Gamma_j$ by Proposition \ref{PROPGM}.
  
 Now we proceed to prove $(ii)$. For each $j=0,\ldots,l$ and $x\in\Gamma_j$ with $F_{\Gamma_j}(x)\neq \emptyset$, $G_{\Gamma_j}$ is locally Lipschitz continuous on $\Gamma_j$ by the above arguments. And $F_{\Gamma_j}$ is locally Lipschitz continuous on $\Gamma_j$ as well. 
  For any $x'\in R(x;t)\cap\overline \Gamma_j$, there exists $y(\cdot)$ satisfying
 \[
 \dot y(s)\in F(y(s))\ \text{a.e.}\ s\in(0,t),\ y(0)=x\ \text{and}\ y(t)=x'.
 \]
 Let $r_0>0$ such that
 \[
 d(y(s),\Gamma_j)<r_0\ \text{for}\ s\in [0,t],\ t\leq \e_j,
 \]
 where $d(\cdot,\Gamma_j)$ is the distance function to $\Gamma_j$. We set 
 \[
 D:=\bigcup_{s\in[0,t]}B(y(s),r_0).
 \]
 Then there exists $M,L>0$ such that
 \[
 \|p\|\leq M,\ \forall\,p\in F(x),\ x\in D,
 \]
 and
 \[
 F_{\Gamma_j}(x_1)\subset F_{\Gamma_j}(x_2)+L\|x_1-x_2\|B(0,1),\ \forall\,x_1,x_2\in D,
 \]
 since $D$ is bounded. Here $F_{\Gamma_j}$ is extended to the domain $D$ by projection to $\Gamma_j$. Let $\e_j$ be small enough such that
 \[
 K\int^t_0 d(\dot y(s),F_{\Gamma_j}(y(s)))ds\leq r_0,
 \]
 where $K:=\exp(Lt)$. 
 By Filippov Existence Theorem \cite[Theorem 3.1.6]{C90}, there exists $z(\cdot)$ such that 
 \[
 \dot z(s)\in F_{\Gamma_j}(z(s))\ \text{a.e.}\ s\in (0,t),\ z(0)=x,
 \]
 and
 \[
 \|z(t)-y(t)\|\leq K\int^t_0 d(\dot y(s),F_{\Gamma_j}(y(s)))ds\leq 2KMt.
 \]
 Let $\tau=\frac{\|z(t)-y(t)\|}{r}$ where $r$ is given in {\bf (H3)}, we define
 \[
 \tilde z(s)=\left\{
 \begin{array}{ll}
  z(s) & \text{for}\ s\in [0,t],\\
  z(t)+\frac{y(t)-z(t)}{\tau}(s-t) & \text{for}\ s\in (t,t+\tau].
 \end{array}
 \right.
 \]
 Then the assumption {\bf (H3)} implies that 
 \[
\dot{ \tilde z}(s)\in F_{\Gamma_j}(\tilde z(s))\ \text{a.e.}\ s\in(0,t+\tau),\ \tilde z(0)=x\ \text{and}\ \tilde z(t+\tau)=y(t)=x'.
 \]
 Therefore, $x'\in R_j(x;t+\tau)$ with
 \[
 t\in [0,\e_j]\ \text{and}\ t+\tau\leq t+\frac{2KMt}{r}\leq \left(1+\frac{2KM}{r}\right)t.
 \]
 Consequently, we conclude the proof by setting $\Delta_j:=1+\frac{2KM}{r}$.
 \end{normalfont}
\end{proofH3}
Now we prove Proposition \ref{contv}.
\begin{proof2} 
\upshape
We split the proof into the following three steps. In {\bf Step $1$} we prove the local Lipschitz continuity on $[0,T]\times \Gamma$, in {\bf Step $2$} we prove the continuity on $[0,T] \times \Gamma$ and finally in {\bf Step $3$} we prove the continuity on $[0,T] \times \mathbb{R}^d$.
\begin{step}{\bf 1- Local Lipschitz continuity on $[0,T]\times \Gamma$.} 
\upshape

\begin{proof} 
	
	For any $t\in[0,T]$, we firstly prove that $v|_{[0,T]\times\Gamma}(t,\cdot)$ is locally Lipschitz continuous on $\Gamma$.  Let $x,z\in\Gamma$ and let $B$ be a ball containing $x,z$. 
	We consider two cases according to the positions of $x,z$, either  $x,z$ belongs to the same hyperplane {\bf Case $1$}, or not {\bf Case $2$}. In {\bf Case $1$} the proof relies strongly on the controllability assumption {\bf (H2)} and is quite standard (see \cite{RSZ}, Theorem $4.5$ (part $1$)). Then, the proof in {\bf Case $2$} is carried out by relying on the result of {\bf Case $1$} and using significantly the cellular structure of our decomposition of $\mathbb{R}^d$.
	\smallskip
	
	{\bf Case 1}: $x,z\in\Hy_j$, for some $j\in \{1,\ldots, q\}$.
	
	The super-optimality implies that for any $\e>0$ there exists $(\bar y,\bar \alpha)$ satisfying \eqref{DSy} with $\bar{y}(t)=z$ such that
	\begin{equation}\label{superoptv}
	v(t,z)\geq \vp(\bar y(T))+\int^T_t \ell(\bar y(s),\bar \alpha(s))ds -\e.
	\end{equation}
	We set
	\[
	h:=\frac{|x-z|}{r_1},\ \xi(s):=x+r_1\frac{z-x}{|z-x|}(s-t),\ \text{for}\ s\in[t,t+h],
	\]
	where $r_1>0$ is as in assumption {\bf (H2)}. 
	Note that $\xi$ is the segment joining $x$ with $z$ during the time interval $[t,t+h]$. Since $|\dot \xi|=r_1$ and $\xi(s)\in\Gamma$ for any $s\in[t,t+h]$, by {\bf (H2)}, there exists $\alpha\in\A$ such that
	\[
	\dot \xi(s)=f(\xi(s),\alpha(s)),\ \text{a.e.}\ s\in(t,t+h).
	\]
	We define
	\[
	\tilde y(s):=\left\{
	\begin{array}{ll}
	\xi(s) & \text{for}\ s\in[t,t+h],\\
	\bar y(s-h) & \text{for}\ s\in[t+h,T],
	\end{array}
	\right.
	\]
	and
	\[
	\tilde \alpha(s):=\left\{
	\begin{array}{ll}
	\alpha(s) & \text{for}\ s\in[t,t+h],\\
	\bar \alpha(s-h) & \text{for}\ s\in[t+h,T].
	\end{array}
	\right.
	\]
	Then $(\tilde y,\tilde \alpha)$ satisfies
	\[
	\dot{\tilde y}(s)=f(\tilde y(s),\tilde \alpha(s)),\ \text{a.e.}\ s\in(t,T),\ \tilde y(t)=x.
	\]
	Let $K$ be a compact containing the support of $\bar{y}(s)$ for $s\in [t,T]$ and denote by $M$ an upper bound for the cost and the dynamic on $K$.  Let $L_ \vp$ denote the Lipschitz constant of $\vp$ on $K$.
	By the sub-optimality of $v$, \eqref{superoptv}, the Lipschitz continuity of $\phi$, we conclude
	\begin{eqnarray*}
		v(t,x)-v(t,z)&\leq& \vp(\tilde y(T))+\int^T_t \ell(\tilde{y}(s),\tilde \alpha(s))ds-\vp(\bar y(T))-\int^T_t \ell(\bar y(s),\bar \alpha(s))ds+\e \\  &=&L_\vp|\bar{y}(T-h)-\bar{y}(T)|+\int^{t+h}_t \ell(\xi(s), \alpha(s))ds-\int^T_{T-h} \ell(\bar y(s),\bar \alpha(s))ds+\e \\
		&\leq& L_ \vp Mh+2Mh+\e=\frac{M(L_ \vp +2)}{r}|x-z|+\e\,\,\, \forall\e\geq 0.
	\end{eqnarray*}
	Then by the arbitrary choice of $\e$, 
	\[
	v(t,x)-v(t,z)\leq \frac{M(L_ \vp +2)}{r}|x-z|,
	\]
	and we conclude the local Lipschitz continuity of $v|_{[0,T]\times\Gamma}(t,\cdot)$ on each $\Hy_j$, for $j=1, \cdots, q$.
	
	\smallskip
	
	{\bf Case 2}: $x,z$ are not on the same hyperplane.
	
	Suppose without loss of generality that $x\in\Hy_{j_1}$, $z\in\Hy_{j_2}$ for some $j_1, j_2 \in \{1,\cdots q\}$. Then we have the following two cases:
\begin{itemize}
\item[(i)]$\Hy_{j_1}\cap \Hy_{j_2}= \emptyset$;
\item[(ii)] $\Hy_{j_1}\cap \Hy_{j_2}\neq\emptyset$. 
\end{itemize}
 We give the proof in case {\bf Case 2} (ii) since the proof in case {\bf Case 2} (i) follows from {\bf Case 1} and {\bf Case 2} (ii). We denote $\Gamma_0=\Hy_{j_1}\cap \Hy_{j_2}$. Consider the projections of $x,z$ on $\Gamma_0$: $\Proj_{\Gamma_0}(x)$ and $\Proj_{\Gamma_0}(z)$. Then we have
	\begin{eqnarray}\label{EqvLips}
	|v(t,x)-v(t,z)|
	&\leq& |v(t,x)-v(t,\Proj_{\Gamma_0}(x))|+|v(t,\Proj_{\Gamma_0}(x))-v(t,\Proj_{\Gamma_0}(z))|+|v(t,\Proj_{\Gamma_0}(z))-v(t,z)|\nonumber\\
	&\leq& \frac{L_ \vp (M+2)}{r}\left(|x-\Proj_{\Gamma_0}(x)|+|\Proj_{\Gamma_0}(x)-\Proj_{\Gamma_0}(z)|+|\Proj_{\Gamma_0}(z)-z|\right).
	\end{eqnarray}
	The last inequality holds because $x,\Proj_{\Gamma_0}(x)\in\Hy_{j_1}$ and $z,\Proj_{\Gamma_0}(z)\in\Hy_{j_2}$.
	
	Now we need to estimate the length of the polyline linking $x$, $\Proj_{\Gamma_0}(x)$, $\Proj_{\Gamma_0}(z)$ and $z$ by the length of the segment joining $x$ with $z$. Since
	\[
	\langle x-\Proj_{\Gamma_0}(x), \Proj_{\Gamma_0}(x)-\Proj_{\Gamma_0}(z)\rangle=0,\ 
	\langle z-\Proj_{\Gamma_0}(z), \Proj_{\Gamma_0}(x)-\Proj_{\Gamma_0}(z)\rangle=0,
	\]
	we have
	\begin{eqnarray*}
		|x-z|^2&=&|x-\Proj_{\Gamma_0}(x)+\Proj_{\Gamma_0}(x)-\Proj_{\Gamma_0}(z)+\Proj_{\Gamma_0}(z)-z|^2\\
		&=& |x-\Proj_{\Gamma_0}(x)|^2+|\Proj_{\Gamma_0}(x)-\Proj_{\Gamma_0}(z)|^2+|\Proj_{\Gamma_0}(z)-z|^2 +2\langle x-\Proj_{\Gamma_0}(x), \Proj_{\Gamma_0}(z)-z\rangle.
	\end{eqnarray*}
	Note that 
	\[
	\left|\langle x-\Proj_{\Gamma_0}(x), \Proj_{\Gamma_0}(z)-z\rangle\right|=|x-\Proj_{\Gamma_0}(x)| |z-\Proj_{\Gamma_0}(z)| |\langle \vec n_1,\vec n_2\rangle|=0,
	\]
	where $n_1, n_2$ are respectively the normal to $\Hy_{j_1}, \Hy_{j_2}$ passing through $x$ and $z$.

	Then
	\begin{eqnarray*}
		|x-z|^2&\geq& |x-\Proj_{\Gamma_0}(x)|^2+|\Proj_{\Gamma_0}(x)-\Proj_{\Gamma_0}(z)|^2+|\Proj_{\Gamma_0}(z)-z|^2 \\
		&\geq&\frac{1}{2}\left(|x-\Proj_{\Gamma_0}(x)|+|\Proj_{\Gamma_0}(x)-\Proj_{\Gamma_0}(z)|+|\Proj_{\Gamma_0}(z)-z|\right)^2.
	\end{eqnarray*}
	Together with \eqref{EqvLips}, it is obtained that
	\begin{equation}\label{loclip1}
	|v(t,x)-v(t,z)|\leq \frac{M(L_\vp+2)}{r}\sqrt{2}|x-z|.
	\end{equation}
	Now given $x\in\Gamma$, we proceed to prove the local Lipschitz continuity of $v|_{[0,T]\times\Gamma}(\cdot,x)$ on $[0,T]$.  For any $t_1,t_2\in [0,T]$, we assume without loss of generality that $t_1<t_2$. For any $\alpha\in\A$, let $y^{\alpha}_{t_2,x}$ be the solution of \eqref{DSy} with the initial condition $y^{\alpha}_{t_2,x}(t_2)=x$.  Denote  by $B$ a ball containing $x$ and let $K$ be a compact containing the support of $y^{\alpha}_{t_2,x}$ respectively in $[t_2,T]$. Let $M$ be an upper-bound for $\ell$ in $K$.  By {\bf (H2)}, let $a\in \As$ such that $f(x,a)=0$. We set
        \[
        \alpha_1(s)=\left\{
        \begin{array}{ll}
        a & \text{for}\ s\in [t_1,t_2),\\
        \alpha(s) & \text{for}\ s\in [t_2,T].
        \end{array}
        \right.
        \]
        Let $y^{\alpha_1}_{t_1,x}$ be the solution of \eqref{DSy} with the initial data $(t_1,x)$ and the control $\alpha_1$. Then we have

        \[
        y^{\alpha_1}_{t_1,x}=\left\{
        \begin{array}{ll}
        x & \text{for}\ s\in[t_1,t_2),\\
        y^{\alpha}_{t_2,x} & \text{for}\ s\in [t_2,T].
        \end{array}
        \right.
        \]
        Therefore
        \begin{eqnarray*}
        &&\left|\vp(y^{\alpha_1}_{t_1,x}(T))+\int^T_{t_1} \ell(y^{\alpha_1}_{t_1,x}(s),\alpha_1(s))ds -\vp(y^{\alpha}_{t_2,x}(T))-\int^T_{t_2} \ell(y^{\alpha}_{t_2,x}(s),\alpha(s))ds\right|\\
        &\leq & \left|\int^{t_2}_{t_1} \ell(x,a)ds\right| \leq M(t_2-t_1).
        \end{eqnarray*}
\end{proof}
\end{step}
\begin{step}{\bf 2- Continuity on $[0,T]\times\Gamma$.}

\smallskip
\upshape
In order to show the continuity of the value function on $[0,T]\times\Gamma$, we need the following lemma on the behavior of controlled dynamics. We refer also to \cite[Lemma 4.3]{RSZ} for an analogous result in the setting of a two-domain partitions of $\mathbb{R}^d$ and control problems with bounded cost and dynamic. The proof is postponed at the end of the proof of Proposition \ref{contv}.

\begin{lmm}\label{LEMcont}
 Assume {\bf(H1)}, {\bf (HF)}, {\bf (HL)}, {\bf (HG)} , {\bf(H2)(ii)}. Let $t \in [0,T], x \in\Gamma$ and $\{x_n\}$ be a sequence such that  $x_n\in \Omega_i$ for some $i \in \{1,\ldots, m\}$ and $x_n \to x$ as $n \to + \infty$. Then for $n$ large enough  there exists two trajectories $\overline y_n,\ \underline y_n$ driven by $F$ and $\overline h_n \to 0,\ \underline h_n \to 0$ as $n \to + \infty$  such that 
 \[
 \overline y_n(t)=x_n,\ \overline y_n(t +\overline h_n)\in\Gamma,\ \overline y_n([t,t +\overline h_n))\subset \Om_i,
 \]
 \[
 \underline y_n(t)\in\Gamma,\ \underline y_n(t +\underline h_n)=x_n,\ \underline y_n((t, t +\underline h_n])\subset \Om_i.
 \]
\end{lmm}
Now we prove that $v$ is continuous at any point of $[0,T]\times\Gamma$.

\begin{proof}
 Taking into account that $v$, restricted on $\Gamma$, is continuous, it is enough to prove that for any $t\in[0,T]$, $x\in\Gamma$,
 \[
 v(t_n, x_n)\ra v(t,x),\ \text{for any}\ t_n\ra t,\ x_n\ra x,\ t_n\in[0,T],\ x_n\in\Om_i,\ i\in \{1,\ldots,m\}.
 \]
 By applying Lemma \ref{LEMcont}, for $n$ large enough there exist , $\overline h_n,\underline h_n$ and $\overline y_n,\ \underline y_n$ driven by $F$ such that
 \[
 \overline y_n(t)=x_n,\ \overline y_n(t+\overline h_n)\in\Gamma,\ \overline y_n([t,t +\overline h_n))\subset\Om_i,
 \]
 \[
 \underline y_n(t)\in\Gamma,\ \underline y_n(t +\underline h_n)=x_n,\ \underline y_n((t,t +\underline h_n])\subset \Om_i.
 \]
 Note that $\overline h_n,\underline h_n\ra 0$ and by the local boundedness of $f$ there exists some constant $M$ such that
 \[
 |\overline y_n(t +\overline h_n)-x|\leq |\overline y_n(t +\overline h_n)-x_n|+|x_n-x|\leq M\overline h_n+|x_n-x|,
 \]
 which implies $\overline y_n(t +\overline h_n)\ra x$. By the same arguments, $\underline y_n(t)\ra x$.
 
 Let $\overline \alpha_n,\underline \alpha_n$ be the corresponding controls for $\overline y_n,\underline y_n$. By the sub-optimality satisfied by $v$, we have
 \[
 v(t_n,x_n)\leq v(t+\overline h_n,\overline y_n(t+\overline h_n))+\int^{t+\overline h_n}_{t_n}\ell(\overline y_n(s),\overline \alpha_n(s))ds\leq v(t+\overline h_n,\overline y_n(t+\overline h_n))+\overline M_n(\overline h_n+t-t_n),
 \]

 \[
 v(t_n-\underline h_n,\underline y_n(t))\leq v(t_n,x_n)+\int^{t_n}_{t_n-\underline h_n}\ell(\underline y_n(s),\underline \alpha_n(s))ds\leq v(t_n,x_n)+\underline M_n\underline h_n.
 \]
 where $\overline M_n=\max_{s\in(t_n,t_n+\overline h_n)}l(\overline y_n(s))$ and $\underline M_n=\max_{s\in(t_n-\underline h_n,t_n)}l(\underline y_n(s))$.
For any $s \in (t_n, t_n + \overline h_n)$ we estimate the cost by the Gronwall lemma and we get
 $$
l(\overline y_n(s))\leq c_l(1+|\overline y_n(s)|)^{\lambda_l}e^{c_f\lambda_ls}.
 $$
Then, since $\overline y_n(s)$ is uniformly bounded in $n$ for $s \in (t_n, t_n + \overline h_n)$ and for large $n$, we get
 $$
 \overline M_n \overline h_n \to 0, \quad  \underline M_n \underline h_n \to 0,  \quad \mbox{as } n \to + \infty.
 $$

 Putting $n\ra+\infty$ and by the continuity of $v|_{[0,T]\times\Gamma}$, we derive
 \[
 \mathop{\lim\,\sup}_{t_n\ra t,\,x_n\ra x}v(t_n,x_n)\leq v(t,x),\ v(t,x)\leq \mathop{\lim\,\inf}_{t_n\ra t,\,x_n\ra x}v(t_n,x_n),
 \]
 which shows the assertion.
\end{proof}
\end{step}
\begin{step}{\bf 3-Continuity in $[0,T]\times\R^d$.}
\upshape
\begin{proof}
The proof follows similar arguments to \cite{RSZ}, Theorem $4.5$ (part $3$) and essentially extends the result to the case of  unbounded cost and dynamic. 
 We consider a bounded subset $B$ of $\Omega_i$. We prove that, given $t\in[0,T]$, for any $\e>0$, there exists $\delta>0$ such that
 \begin{equation}\label{claimcont}
 |v(t,z)-v(t,x)|<\e,\ \text{for any}\ x,z \in B \mbox{ such that } |x-z|< \delta.
 \end{equation}
We denote by $K$ a compact set containing the support of any integral curve
of $F$, starting at $B$, and defined in $[0, T]$.
 In the following, we denote by $L$ and $L_\phi$ the Lipschitz constant respectively of $f_i$,$\ell_i$ in $K\times \As$  and of $\phi$ in $K$. \\
 By the super-optimality of $v$, there exists $(\bar y_x,\bar \alpha)$ satisfying \eqref{DSy} with $\bar{y}_x(t)=x$ such that
 \begin{equation}\label{Vconteq1}
 v(t,x)\geq v(t+h,\bar y_x(t+h))+\int^{t+h}_t\ell(\bar y_x(s),\bar \alpha(s))ds-\frac{\epsilon}{2},\ \forall\,h\geq 0.
 \end{equation}
 Let $\bar y_z$ be the solution of \eqref{DSy} with the control $\bar \alpha$ and the initial condition $\bar y_z(t)=z$. 
 By the sub-optimality of $v$,
 \begin{equation}\label{Vconteq2}
 v(t,z)\leq v(t+h,\bar y_z(t+h))+\int^{t+h}_t\ell(\bar y_z(s),\bar \alpha(s))ds,\ \forall\,h\geq 0.
 \end{equation}
 Now define
 \[
 \tilde T:=\inf\{s\,:\, \bar y_x(s)\not\in \Om_i\ \text{or}\ \bar y_z(s)\not\in \Om_i,\ s\in[t,T]\}.
 \]
 If $\tilde T= T$, then $\bar y_x$ and $\bar y_z$ stay in $\Om_i$ during $(t,T)$. Thus, $\bar y_x$ and $\bar y_z$ are always driven by $F_i$ and
 the Gronwall lemma implies that
 \[
 |\bar y_x(s)-\bar y_z(s)|\leq e^{Ls}|x-z|,\ \forall\,s\in[t,T],
 \]
 where $L$ depends on the compact $K$ where the support of $y$ lies.
 Then by taking $h=T-t$ in \eqref{Vconteq1} and \eqref{Vconteq2} we obtain
 \begin{eqnarray*}
 &&v(t,z)-v(t,x)\\
 &\leq& \vp(\bar y_z(T))+\int^T_t \ell_i(\bar y_z(s),\bar \alpha(s))ds -\vp(\bar y_x(T))-\int^T_t \ell_i(\bar y_x(s),\bar \alpha(s))ds+\frac{\e}{2}\\
 &\leq& (L_ \vp+TL)e^{LT}|x-z|+\frac{\e}{2}.
 \end{eqnarray*}
 The assertion holds true by taking any $\delta>0$ with $\delta\leq\frac{e^{-LT}}{2(L_\vp+TL)}\e$.
 
 \smallskip
 Otherwise, if $\tilde T<T$, we still have
 \[
 |\bar y_x(s)-\bar y_z(s)|\leq e^{Ls}|x-z|,\ \forall\,s\in[t,\tilde T].
 \]
 By taking $h=\tilde T-t$ in \eqref{Vconteq1} and \eqref{Vconteq2} we obtain
 \begin{eqnarray}\label{vzmenovx}
  &&v(t,z)-v(t,x) \nonumber\\
 &\leq& v(\tilde T,\bar y_z(\tilde T))+\int^{\tilde T}_t \ell_i(\bar y_z(s),\bar \alpha(s))ds -v(\tilde T,\bar y_x(\tilde T))-\int^{\tilde T}_t \ell_i(\bar y_x(s),\bar \alpha(s))ds+\frac{\e}{2}\nonumber\\
 &\leq& v(\tilde T,\bar y_z(\tilde T))-v(\tilde T,\bar y_x(\tilde T))+TLe^{LT}|x-z|.
 \end{eqnarray}
 Note that $\tilde T<T$ implies that $\bar y_x(\tilde T)\in\Gamma$ or $\bar y_z(\tilde T)\in\Gamma$. Without loss of generality suppose that $\bar y_x(\tilde T)\in\Gamma$. By the continuity of $v(\tilde T,\cdot)$ on $\Gamma$, there exists $\delta_1>0$ such that
 \begin{equation}\label{contgamma}
 |v(\tilde T,x')-v(\tilde T,\bar y_x(\tilde T))|<\frac{\e}{2},\ \forall\,x'\in\R^d,\ |x'-\bar y_x(\tilde T)|<\delta_1.
 \end{equation}
Take $\delta>0$ such that
 \[
 \delta<\min\{e^{-L\tilde{T}}\delta_1,\frac{e^{-LT}}{2TL}\e\}.
 \]
 Then, for any  $|x-z|<\delta$, we have
 \begin{equation}\label{deltauno}
 |\bar y_x(\tilde T)-\bar y_z(\tilde T)|<\delta_1, 
 \end{equation}
 and
 \begin{equation}
 \label{xmenozpiccolo}
TLe^{LT}|x-z|<\frac{\e}{2}.
 \end{equation}
Then, by \eqref{deltauno} and \eqref{contgamma} with $x'=\bar{y}_z(\tilde{T})$, we get
 \begin{equation}\label{contgammafine}
 v(\tilde T,\bar y_z(\tilde T))-v(\tilde T,\bar y_x(\tilde T))<\frac{\e}{2}
 \end{equation}
 and the claim \eqref{claimcont} follows by coupling \eqref{xmenozpiccolo}, \eqref{contgammafine} and \eqref{vzmenovx}.
\end{proof}
\end{step}
\end{proof2}
Finally we prove Lemma \ref{LEMcont}.
\begin{proof3}
\upshape
Without loss of generality, we suppose that   $x \in \Hy_j$ for just one $j\in\{1,\ldots,q\}$. Since $x_n \to x$, we have that for $n$ large enough, $x_n\in\Om_i$ such that $\bar{\Omega}_i \cap \Hy_j \neq \emptyset$. Let $n\geq 0$ be fixed.

Define $g_j:\R^d\ra\R$ as
\[
g_j(x):=\left\{
\begin{array}{ll}
d_{\Hy_j}(x) & \text{if}\ (x-\Proj_{\Hy_j}(x))\cdot \vec n_j\geq 0,\\
-d_{\Hy_j}(x) & \text{otherwise},
\end{array}
\right.
\]
where $n_j$ denotes the normal vector to each $\Hy_j$ as defined in subsection \ref{secass}.
 By applying \cite[Lemma 4.3]{RSZ} for  $\K=\{x\}$ of $\Hy_j$ and $\Om_i$, there exists  $S>0$ such that   for any $n$ large enough, there exists two trajectories $\overline y_n,\ \underline z_n$ driven by $F$ and $\overline t_n,\ \underline t_n$ less than $Sg_j(x_n)$ with
 \[
 \overline y_n(t)=x_n,\ \overline y_n(t+\overline t_n)\in\Hy_j,
 \]
 \[
 \underline z_n(t)\in\Hy_j,\ \underline z_n(t+\underline t_n)=x_n.
 \]
 Note that $\overline t_n, \underline t_n \to 0$ as $n \to + \infty$. 
 For $\overline y_n$, we take $\overline h_n:=\min\{s\,:\overline y_n(s)\in\Gamma,\ s\in[t,t+\overline t_n]\}$, then we have
 \[
 \overline y_n(t)=x_n,\ \overline y_n(t+\overline h_n)\in\Gamma,\ \overline y_n([t,t+\overline h_n))\subset \Om_i.
 \]
 For $\underline z_n$, we take $\tau_n:=\sup\{s\,:\,\underline z_n(s)\not\in\Om_j,\ s\in[t,t+\underline t_n]\}$, $\underline h_n=\underline t_n-\tau_n$ and $\underline y_n(\cdot)=\underline z_n(\cdot-\tau_n)$, then we have
 \[
 \underline y_n(t)\in\Gamma,\ \underline y_n(t+\underline h_n)=x_n,\ \underline y_n((t,t+\underline h_n])\subset \Om_i.
 \]
 and the claim follows by noting that $\underline h_n\leq \underline \tau_n, \overline h_n \leq \overline \tau_n$ and then $\underline h_n, \overline h_n \to 0$ as $n \to + \infty$. 

 Finally, if $x \in \Hy_{j_1}\cap\Hy_{j_2}$ for some $j_1, j_2 \in \{1, \ldots q\}$, there exist the desired $ S,\ \overline y_n,\ \underline y_n,\ \overline h_n,\ \underline h_n$ with $\overline h_n,\ \underline h_n\leq S\min\{g_{j_1}(x_n),g_{j_2}(x_n)\}$.
\end{proof3}

\section{Appendix B}
\subsection{Some background in non smooth analysis: trajectories and invariance}
We recall here some fundamental results which we need in the characterization of the sub-optimality.
The first proposition states the existence of smooth trajectories for a given initial data,
namely, initial point and initial velocity.
The proof is analogous to the proof of Proposition $4.1$ of \cite{HZ} and we omit it.
\begin{lmm}\label{trajectories}
	Assume {\bf (H1)}, {\bf(HF)}, {\bf(HL)}, {\bf(HG)}, {\bf (H3)}. Then, for any $k\in \{0,\cdots m+l\}$
	such that $A_{\M_k}$ has nonempty images, for every $(t,x) \in [0,T] \times \M_k$ and any $a \in A_{\M_k}(x)$ there exist $\tau>0$, a measurable control map $\alpha \,: \,(t-\tau,t+\tau] \to \mathcal{A}$, a measurable function $r \, : \, (t-\tau, t+\tau]\to [0,+\infty)$ and  $(y(\cdot),\eta(\cdot))\in C^1((t-\tau,t+\tau]),\, y(s) \in \M_k$ for any $s \in (t-\tau, t+\tau]$, such that
	\[
	\dot y(s)=f(y(s),\alpha(s)),\ \dot \eta(s)=-\ell(y(s),\alpha(s))-r(s)
	\]
	and
	\[
	y(t)=x,\ \dot y(t)=f(x,a), \ \eta(t)=0, \ \dot \eta(t)=l(x,a).
	\]
\end{lmm}
We recall the notion of  proximal subgradient,  proximal normal cone and its relation with the proximal subgradients. We refer to \cite{CLSW} for more details.\\
\begin{dfntn}
	Let $\omega\, : \, \R^d \rightarrow \R \cup \{+\infty\}$ be a given l.s.c function. A viscosity subgradient $\eta \in \R^d$ of $\omega$ at $x \in \mbox{ dom } \omega$ is called a proximal subgradient of $\omega$ at $x$ if  for some $\sigma >0$ the test function $g\, : \, \R^d \rightarrow \R$ can be taken as
	$$
	g(y):=<\zeta, y-x>-\sigma |y-x|^2, \quad \forall y \in \R^d.
	$$
	We denote the set of all proximal subgradients at $x$ by $\partial_P \omega(x)$.
\end{dfntn}
Let $\mathcal{B}\subseteq \R^d$ be a locally  closed set. For any $x \in \mathcal{B}$ a vector $\eta \in \R^d$ is called proximal normal to $\mathcal{B}$ at $x$ if there exists $\sigma=\sigma(x,\eta) >0$ so that
$$
\frac{|\eta|}{2\sigma}|x-y|^2\geq <\eta, y-x>\quad \forall y \in \mathcal{B}.
$$
The \textit{Proximal normal cone} to $\mathcal{B}$ at $x$ is the set of all such vectors $\eta$. We denote it by $\mathcal{N}_\mathcal{B}^P(x)$.  \\
When $\mathcal{B}=\Ep(\omega)$ where $\omega \, : \, \R^d \rightarrow \R \cup \{+\infty\}$ is a l.s.c function, then for each $x \in \mbox{ dom } \omega$, the following relation holds:
\begin{equation}\label{relation}
\xi \in \partial_P\omega(x) \Longleftrightarrow (\xi, -1) \subseteq \mathcal{N}_\mathcal{B}^P(x,\omega(x)), \quad \forall x \in \mbox{ dom } \omega.
\end{equation}
Finally, we present a useful criterion for strong invariance adapted to smooth manifolds. For the proof we refer to \cite{HZ} Proposition $4.2$ or \cite{HZW} Lemma $3.4$.
\begin{lmm}\label{strinv}
	Suppose $M \subseteq \R^d$ is locally closed, $\mathcal{B} \subseteq \R^d$ is closed with $\mathcal{B} \cap \bar{M} \neq \emptyset$ and $\Gamma\, : \, \bar{M} \rightsquigarrow \R^d$ is locally Lipschitz and locally bounded. 
	
	Let $r>0$ and assume that there exists $c=c(r)>0$ such that
	$$
	\sup_{\nu \in \Gamma(x)} <x-s, \nu> \leq c \mbox{ dist }_{\mathcal{B}\cap \bar{M}}(x)^2, \quad \forall x \in M \cap B_r, \, \forall s \in \mbox{ proj }_{\mathcal{B}\cap \bar{M}}(x).
	$$
	Then for any absolutely continuous arc $\gamma\, : \, [0,T] \rightarrow \bar{M}$ that satisfies
	$$
	\dot{\gamma} \in \Gamma(\gamma) \quad \mbox{ a.e. on } [0,T] \quad \mbox{ and } \gamma(t) \in M \cap B_r \quad \forall t \in (0,T),
	$$
	the following estimate holds true
	$$
	\mbox{ dist }_{\mathcal{B}\cap \bar{M}}(\gamma(t))\leq e^{ct}\mbox{ dist }_{\mathcal{B}\cap \bar{M}}(\gamma(0)) \quad \forall t \in [0,T].
	$$
\end{lmm}

\subsection{Proof of Theorem \ref{CHARAsubback}}
\begin{proof}
	First we prove the implication (i) $\Rightarrow$ (ii), that is, if $u:[0,T]\times\R^d\ra\R$ is a lsc function satisfying  the sub-optimality, then  is the bilateral subsolution to \eqref{HJBOmi}-\eqref{HJBHGamma}. 
	
	Let $(t_0,x_0)\in (0,T)\times\R^d$, $k\in\{0,\ldots,m+l\}$ with $x_0\in\M_k$ and $\phi\in C^1((0,T)\times\M_k)$ such that $u|_{\M_k}-\phi$ attains a local minimum at $(t_0,x_0)$.
	We assume  $A_{\M_k}(x_0) \neq \emptyset$, otherwise the claim is trivial. For any $a\in A_{\M_k}(x_0)$, since $G_{\M_k}$ is locally Lipschitz continuous, by Lemma \ref{trajectories}, there exist $\tau>0$, $(y(\cdot),\eta(\cdot))\in C^1((t_0-\tau,t_0]), y(s) \in \M_k$ for any $s \in (t_0-\tau, t_0]$ and a measurable control map $\alpha \, : \, (t_0-\tau, t_0] \to \mathcal{A}$ such that
	\begin{equation}\label{hp1ns}
	\dot y(s)=f(y(s),\alpha(s)),\ \dot \eta(s)\leq -\ell(y(s),\alpha(s))
	\end{equation}
	and
	\begin{equation}\label{hp2ns}
	 y(t_0)=x_0,\ \dot y(t_0)=f(x_0,a),\ \eta(t_0)=0, \ \dot \eta(t_0)=l(x_0,a).
	\end{equation}
	Take $\bar y,\bar \alpha$ satisfying \eqref{DSyback} on $(0,t_0-\tau)$ such and $\bar y(t_0-\tau)=y(t_0-\tau)$ and remark that $\tilde y, \tilde \alpha$ where $\tilde y=\bar y 1_{[0,t_0-\tau)}+y1_{[t_0-\tau, t_0]}$, $\tilde \alpha=\bar \alpha 1_{[0,t_0-\tau)}+\alpha1_{[t_0-\tau, t_0]} $ satisfy \eqref{DSyback} on $(0,t_0)$ with $\tilde y(t_0)=x_0$.
	Therefore, since $u$ satisfies the suboptimality and then by Proposition \ref{DPPback}, $u$ satisfies the backward sub-optimality, we have that 
	\begin{equation}\label{dppns}
	u(t_0,x_0)\geq u(t_0-h,y(t_0-h))-\eta(t_0-h)+\eta(t_0) \quad \forall h \in [0,\tau).
	\end{equation}
	Since $(t_0,x_0)$ is a local minimum of $u|_{\mathcal{M}_k}-\phi$ we have
	\begin{equation}\label{minns}
	u(t_0,x_0)-\phi(t_0,x_0)\leq u(t_0-h,y(t_0-h))-\phi(t_0-h,y(t_0-h))\quad \forall h \in [0,\tau),
	\end{equation}
	and by combining \eqref{dppns} and \eqref{minns} we get
	\begin{equation}\label{phins}
	\phi(t_0-h,y(t_0-h))-\phi(t_0,x_0)-\eta(t_0-h)+\eta(t_0)\leq 0 \quad \forall h \in [0,\tau).
	\end{equation}
	By \eqref{phins}, \eqref{hp1ns} and \eqref{hp2ns} we get
	\[
	-\partial_t\phi(t_0,x_0)-D\phi(t_0,x_0)\cdot f(x_0,a)-\ell(x_0,a)\leq 0,\ \forall\,a\in A_{\M_k}(x_0).
	\]
	from which we conclude
	\[
	-\partial_t\phi(t_0,x_0)+H_{\M_k}(x_0,D\phi(t_0,x_0))\leq 0,
	\]
	that is, $u$ is a bilateral subsolution to \eqref{HJBOmi}-\eqref{HJBHGamma}.
	
	\smallskip
	
	Now we prove the implication (ii) $\Rightarrow$ (i). We will prove that,  if  $u:[0,T]\times\R^d\ra\R$ is a lsc bilateral subsolution to \eqref{HJBOmi}-\eqref{HJBHGamma} and if $(y(\cdot),\eta(\cdot))$ satisfies \eqref{DIyeta} on some $[a,b]\subset [0,T]$,
		 it holds that
		\begin{equation}\label{nobackfine}
		u(a,y(a))-\eta(a)\leq u(b,y(b))- \eta(b).
		\end{equation}
	Note that  the suboptimality follows from \eqref{nobackfine} by the same arguments used in \textbf{Step 4} of the proof of Theorem \ref{CHARAsub}.
	We recall that for any $(x,t) \in \R^d \times [0,T]$, we denote by $\mathcal{S}_t^T(x)$ any trajectory satisfying \eqref{DSy}.
We divide the proof into three steps.
In \textbf{Step 1} we treat the case of trajectories staying on one subdomain in Proposition \ref{PROsubmk2}. Then in \textbf{Step 2} we deal with the regular trajectories and finally in \textbf{Step 3} we deal with non regular trajectories.
	
	\begin{step}{{\bf 1}-Trajectories in a subdomain.}
	\upshape
	
	\begin{prpstn}\label{PROsubmk2}
		Let {\bf(H1)}, {\bf (HF)}, {\bf (HL)}, {\bf(HG)} hold. Let $u$ be a lsc bilateral subsolution to \eqref{HJBOmi}-\eqref{HJBHGamma}, $k\in\{0,\ldots,m+l\}$ and $(y(\cdot),\eta(\cdot))$ satisfying \eqref{DIyeta} on some $[a,b]\subset [0,T]$ with $y(s)\in\M_k$ for $s\in[a,b]$.
		Then it holds that
		\begin{equation}\label{onlyonmk}
		u(a,y(a))-\eta(a)\leq u(b,y(b))- \eta(b).
		\end{equation}
	\end{prpstn}
	\begin{proof}
\upshape
	We consider the backward augmented dynamic defined for any $x \in \M_k$ as follows 
	$$
  G_{\M_k}(x)=
\{ -(f(x,a),
 l(x,a)+r), a \in \A_{\M_k}(x), 0\leq r \leq b(x,a)\}.
$$
Note that the mapping $G_k$ has convex compact images  by \textbf{(HG)},  has nonempty images and is locally Lipschitz by Proposition \ref{PROPGM}.  Set $M_k=\R\times \M_k\times \R^2$ and define  $$\mathcal{G}_k(t,x,z,w)= \{-1\}\times G_{\M_k}(x)\times \{0\}, \quad \forall (t,x,z,w) \in [0,T]\times \M_k \times \R^2. $$ 
Note that $M_k$ is an embedded manifold of $\R^{d+3}$ and $\mathcal{G}_k$ satisfies the same assumptions of $G_{\M_k}$.   Consider the closed set $\mathcal{S}_k=\mathcal{E}p(u_k)$ where  $\forall (t,x,z) \in [0,T] \times \overline{\M}_k \times \R$
	$$
 \left\{
 \begin{array}{ll}
  u_k(t,x,z)=u(x)+z &\mbox{ if } x \in \overline{\M}_k,\\
  +\infty & \mbox{ otherwise }.
 \end{array}
 \right.
$$
	Note that, if $u$ is  a l.s.c. bilateral subsolution of \eqref{HJBOmi}-\eqref{HJBHGamma}, the following  hold
	\begin{equation}\label{claimpr}
	\sup_{\nu \in \mathcal{G}_k(t,x,z,w)} (\eta,\nu) \leq 0 \quad \forall (t,x,z,w) \in \mathcal{S}_k, \forall \eta \in \mathcal{N}_{\mathcal{S}_k}^P(t,x,z,w).
	\end{equation}
	Indeed, if $\mathcal{S}_k=\emptyset$, it holds by vacuity. Otherwise, take $(t,x,z,w) \in \mathcal{S}_k$ and a proximal normal $(\xi, -p) \in \mathcal{N}_{\mathcal{S}_k}^P(t,x,z,w)$. Therefore we have $p \geq 0$ since $\mathcal{S}_k$ is the epigraph of a function. Consider $p>0$, then $w=u_k(t,x,z)$ and by \eqref{relation} we have 
	$$
	\frac{1}{p}\xi \in \partial_P u_k(t,x,z) \subseteq \partial_p u_k(t,x) \times \{1\},
	$$
	and then for any $\nu \in \mathcal{G}_k(t,x,z,w)$, for some $\alpha \in \mathcal{A}_{\M_k}(x)$, $r\geq 0$ and for $(\theta, \zeta) \in \partial_P u_k(t,x)$ we get
	
	\begin{eqnarray}
	<(\xi,-p),\nu>&=&p(-\theta-<\zeta,f(x,\alpha)>-\ell(x,\alpha)-r)\nonumber \\&\leq  &-\theta+ \sup_{\alpha \in \mathcal{A}_{\M_k}} \{-<\zeta, f(x,\alpha)> -\ell(x,\alpha)\}.
	\end{eqnarray}
	
	Since $u$  is subsolution  of \eqref{HJBOmi}-\eqref{HJBHGamma} and $\nu \in \mathcal{G}_k(t,x,z,w)$ is arbitrary, we can take the supremum over $v$ and obtain the desired inequality.
	If $p=0$, we use the Rockafellar's horizontal Theorem (cf. \cite{rock}, Theorem $11.30$) and the continuity of $\mathcal{G}_k$ to obtain \eqref{claimpr} for any $\eta$.\\
	Now take $[a,b]\subseteq [0,T]$ and $y \in \mathcal{S}_a^b(x)$ as in the statement.
	Let $r>\tilde{r}>0 $ be large enough so that $y([a,b])\subseteq B_r$ and
	$$
	\sup_{X \in M\cap B_{\tilde r}} |\mbox{ proj }_{\mathcal{S}_k}(X)|<R.
	$$
	Let $L_k$ be the Lipschitz constant for $\mathcal{G}_k $ on $M_k \cap B_r$. Note that $\mathcal{S}_k \cap \bar{M}=\mathcal{S}_k$ and $X-\mbox{ proj }_{\mathcal{S}_k}(X) \in \mathcal{N}_{\mathcal{S}_k}^P(X)$ for any $X \in \mathcal{S}_k$. Therefore, \eqref{claimpr} implies \eqref{relation} with $c=L_k$. Then, by Lemma \ref{strinv} we have that for any absolutely continuous arc $\gamma\, : \, [a,b] \rightarrow \overline{M}_k$ which satisfies
	$$
	\dot{\gamma} \in \mathcal{G}_k(\gamma) \mbox{ a.e. on } [a,b] \quad \mbox{ and } \gamma(s) \in M_k \cap B_r \quad \forall s \in (a,b),
	$$
	the following holds
	\begin{equation}\label{dist}
	\mbox{dist }_{\mathcal{S}_k}(\gamma(s)) \leq e^{L_ks} \mbox { dist }_{\mathcal{S}_k}(\gamma(a)) \quad \forall s \in [a,b].
	\end{equation}
	Let $\alpha \in \mathcal{A}(x)$ be the control associated to the trajectory $y$. Note that the absolutely continuous arc defined by
	$$
	\gamma_y(s)=\left((a+b-s,y(a+b-s),-\int_{a}^s \ell(y(a+b-t), \alpha(a+b-t) dt ,u(b, y(b))\right), \quad \forall s \in [a,b]
	$$
	fulfills the condition for \eqref{dist} to holds.
	Finally, since $\gamma_y(a) \in \mathcal{S}_k$, \eqref{dist} implies that $\gamma_y(b) \in \mathcal{S}_k$ which leads to \eqref{onlyonmk} after some algebraic steps.
\end{proof}
	\end{step}
		\begin{step}{{\bf 2}-Regular trajectories.}
		\upshape
	 We take $[a,x] \in [0,T]\times \R^d$, and $y \in \mathcal{S}^T_a(x)$ for which there
	exists a partition of $[a,T], a= t_0 < t_1 < \cdots < t_n < t_{n+1} = T$, so
	that for any  $l \in \{0,\dots, n\}$ we can find $k$ such that $y(s)\in \M_k$ on
	$(t_l, t_{l+1})$. Then by applying Proposition \ref{PROsubmk2} on each subinterval  $(t_l, t_{l+1})$, we get for any $b \subseteq [a,T]$
	\[
	u(a,y(a))-\eta(a)\leq u(b,y(b))- \eta(b).
	\]
	\end{step}
	\begin{step}{{\bf3}- Non regular trajectories.}
\upshape
We use the following lemma, which is proved in \cite{HZW}, Lemma $3.3$.
	
	\begin{lmm}\label{tretre}
	Assume {\bf(H1)}, {\bf (Hf)}, {\bf (Hl)}, {\bf(Hg)}, {\bf (H3)}. Let $(t,x) \in [0,T] \times \R^d$ and $y(\cdot) \in \mathcal{S}_{t}^T(x)$ be given, then for any $\e>0$ and $\tau \in [t,T]$ we can find $x_{\e} \in B(x,\e)$, $t_{\e} \in (t-\e, t+\e) \cap [0, \tau]$ and $y_{\e} \in S_{t_{\e}}^{\tau}(x_\e)$  that verifies $y_{\e}(\tau)=y(\tau)$ and that is regular in the following sense:\\
		There exists a partition of $[t,\tau]$, $\{t=t_0<t_1<\cdots<t_n<t_{n+1}=\tau\}$, so that for any $l \in \{0,\cdots, n\}$ we can find $k$ such that $y_\e(s) \in \M_k$ on $(t_l, t_{l+1})$.
	\end{lmm}
	
	Then we treat non regular trajectories by  applying Lemma \ref{tretre} as follows.
	Let $(a, x) \in [0, T]\times \R^d$, $b \in [a, T]$ and $y \in \mathcal{S}^T_a(x)$ and take a sequence $\e_n \subseteq (0,1)$ with $\e_n \to 0$. Let $x_n \in \R^d, t_n \in [0,T]$ and $y_n \in \mathcal{S}_{t_n}^{b}(x_n)$ given by Lemma \ref{tretre} with $\e=\e_n$. Then we have that
	\[
	u(t_n,x_n)-\eta(t_n)\leq u(b,y(b))-\eta(b).
	\]
	Then, since $x_n \to x, t_n \to a $, and by the lower semi-continuity of $u$,  we get \eqref{nobackfine} and we conclude the proof.
	\end{step}
	
\end{proof}


\begin{thebibliography}{ABC}
\renewcommand{\baselinestretch}{0.7}
\bibitem{ACCT}
{\sc Y.~Achdou, F.~Camilli, A.~Cutri AND N.~Tchou}, {\em \
Hamilton-Jacobi equations constrained on networks}, Nonlinear Differential Equations Appl., 20/3 (2013), 413-445.
\bibitem{AC}
{\sc J.-P.~Aubin AND A.~Cellina}, {\em \
Differential inclusions}, Comprehensive Stud. Math. 264, Springer, Berlin, 1984.
\bibitem{BBC}
{\sc G.~Barles, A.~Briani AND E.~Chasseigne}, {\em \ A Bellman approach for two-domains optimal control problems in $\R^N$},
ESAIM Control Optim. Calc. Var., 19 (2013), 710-739.
\bibitem{BBC2}
{\sc G.~Barles, A.~Briani AND E.~Chasseigne}, {\em \ A Bellman approach for regional optimal control problems in $\R^N$},
SIAM J. Control Optim. 52/3 (2014), 1712-1744.
\bibitem{BCS}
{\sc G.~Barles,  AND E.~Chasseigne}, {\em (Almost) everything you always wanted to know about deterministic control problems in stratified domains}, Networks and heterogenous media 10 (2015), 809-836.
\bibitem{BC}
{\sc M.~Bardi AND I.~Capuzzo-Dolcetta}, {\em \
Optimal Control and Viscosity Solutions of Hamilton-Jacobi-Bellman Equations}, Systems Control Founda. and Appl., Birkh\"{a}user, Boston, 1997.
\bibitem{BD}
{\sc A.~Briani AND A.~Davini}, {\em \
Monge solutions for discontinuous Hamiltonians}, ESAIM Control Optim. Calc. Var., 11 (2005), 229-251.
\bibitem{BH}
{\sc A.~Bressan AND Y.~Hong}, {\em \
Optimal control problems for control systems on stratified domains}, Network and Heterogeneous Media, 2/2 (2007), 313-331.
\bibitem{BW}
{\sc R.C.~Barnard AND P.R.~Wolenski}, {\em \
Flow invariance on stratified domains}, Set-Valued Variational Anal., 21 (2013), 377-403.
\bibitem{CM}
{\sc F.~Camilli AND C.~Marchi}, {\em \
A comparison among various notions of viscosity solutions for Hamilton-Jacobi equations on networks}, J. Math. Anal. Appl. 407 (2013), 112-118.
\bibitem{CS}
{\sc F.~Camilli AND A.~Siconolfi}, {\em \
Hamilton-Jacobi equations with measurable dependence on the state variable}, Advances in Differential Equations, 8/6 (2003), 733-768.
\bibitem{C90}
{\sc F.H.~Clarke}, {\em \
Optimization and Nonsmooth Analysis}, SIAM, Philadelphia, 1990.

\bibitem{rock}
{\sc F. H. ~Clarke}, {\em \ Functional analysis, calculus of variations and optimal control}, Volume 264 of
Graduate Text in Mathematics. Springer, 2013.

\bibitem{C} {\sc F.H. ~Clarke}, {\em \ Optimization and Nonsmooth Analysis}, Society for Industrial Mathematics,
1990.
\bibitem{CLSW} {\sc F.H. ~Clarke, Yu.S. ~Ledyaev, R.J. ~Stern AND P.R. ~Wolenski}, {\em \ Nonsmooth Analysis
and Control Theory}, Graduate Texts in Mathematics 178, Springer-Verlag, New York.
\bibitem{DF}
{\sc G.~Dal Maso AND H.~Frankowska}, {\em \
Value function for Bolza problem with discontinuous Lagrangian and Hamilton-Jacobi inequalities}, ESAIM Control Optim. Calc. Var., 5 (2000), 369-394.
\bibitem{F88}
{\sc A.F.~Filippov}, {\em \
Differential Equations with Discontinuous Right-Hand Sides}, Kluwer Academic Publishers, 1988.
\bibitem{F}
{\sc H.~Frankowska}, {\em \
Lower semicontinuous solutions of Hamilton-Jacobi-Bellman equations}, SIAM J. Control Optim., 31/1 (1993), 257-272.
\bibitem{FP}
{\sc H.~ Frankowska and S.~ Plaskacz}, {\em \
 Semicontinuous solutions of Hamilton-Jacobi-Bellman equations
with degenerate state constraints}, J. Math. Anal. Appl., 251/2 (2000), 818-838.
\bibitem{HZ}
{\sc C.~Hermosilla AND H.~Zidani}, {\em \
Infinite horizon problems on stratifiable state constraints sets}, J. Differential Equations, 258/4 (2015), 1430-1460.
\bibitem{HZW} 
{\sc C.~Hermosilla, H.~Zidani AND P. Wolenski} {\em \
}, The mayer and minimum time problems with stratified state constraints, Preprint $2017$.
\bibitem{I}
{\sc H.~Ishii}, {\em \
A boundary value problem of the Dirichlet type for Hamilton-Jacobi equations}, Ann. Sc. Norm. Sup. Pisa (IV) 16 (1989), 105-135.
\bibitem{IM}
{\sc C.~Imbert, AND R.~Monneau}, {\em \
Flux-limited solutions for quasi-convex Hamilton-Jacobi equations on networks}, Hamilton-Jacobi equations on networks.
\bibitem{IMZ}
{\sc C.~Imbert, R.~Monneau AND H.~Zidani}, {\em \
A Hamilton-Jacobi approach to junction problems and application to traffic flows}, ESAIM Control Optim. Calc. Var., 19/1 (2013), 129-166. 
\bibitem{RSZ}
{\sc Z.~Rao, A.~Siconolfi AND H.~Zidani}, {\em \
Stationary Hamilton-Jacobi-Bellman equations on multi-domains}, J. Differential Equations, 257/11 (2014), 3978-4014.
\bibitem{RZ}
{\sc Z.~Rao AND H.~Zidani}, {\em \
Hamilton-Jacobi-Bellman equations on multi-domains}, Control and Optimization with PDE Constraints, Internat. Ser. Numer. Math. 164, Birkh\"{a}user Basel (2013), 93-116.
\bibitem{S}
{\sc P.~Soravia}, {\em \
Boundary Value Problems for Hamilton-Jacobi Equations with Discontinuous Lagrangian,} Indiana Univ. Math. J., 51/2 (2002), 451-477.
\bibitem{WZ}
{\sc P.~Wolenski AND Y. ~ Zhuang}, {\em \ Proximal analysis and the minimal time function.} SIAM J. Control Optim., 36/3 (1998), 1048-1072 (electronic).
\end{thebibliography}
\end{document}